\documentclass[11pt,a4paper,headinclude,footinclude,fleqn,reqno]{amsart}                 
\usepackage[T1]{fontenc}                   
\usepackage[utf8]{inputenc}                 
\usepackage[english]{babel}       
\usepackage{graphicx}                      %
\usepackage[font=small]{quoting}            
\usepackage{caption}  
\usepackage{float}
\usepackage[colorlinks=true,linkcolor=blue]{hyperref}
                  
\usepackage[top=.8in,bottom=1.2in,left=1.5in,right=1.5in]{geometry}
\usepackage{verbatim}
\usepackage{color} 
\usepackage{picture}

\usepackage{amsmath,amsthm,amsfonts,amssymb}
\usepackage{comment}
\usepackage{hyperref}
\hypersetup{colorlinks,linkcolor={blue},citecolor={blue},urlcolor={red}}  

\usepackage{accents}

\newtheorem{thm}{Theorem}[section] 
\newtheorem{lem}[thm]{Lemma} 
\newtheorem{cor}[thm]{Corollary} 
\newtheorem{prop}[thm]{Proposition}

\theoremstyle{definition} 
\newtheorem{rem}[thm]{Remark} 
\theoremstyle{remark}

\theoremstyle{definition}

\def\O{\Omega}

\def\S{\Sigma} 
\def\n{\nabla}

\def\p{\partial}

\def\a{\alpha}

\def\n{\nabla}

\def\O{\Omega}
\def\p{\partial}

\def\a{\alpha}

\def\k{\kappa}

\def\De{\Delta}
\def\n{\nabla}
\def\<{\langle}
\def\>{\rangle}

\def\De{\Delta}
\def\n{\nabla}

\def\RR{\mathbb{R}}

\def\ff{\mathcal{F}}
\def\C{\mathcal{C}}
\def\SS{\mathbb{S}}

\def\La{\Lambda}

\def\O{\Omega}
\def\p{\partial}

\def\a{\alpha}

\def\v{\varphi}
\def\ve{\varepsilon}

\def\wh{\widehat}

\def\wt{\widetilde}

\def\R{\mathbb{R}}

\def\I{\mathcal{I}}
 \def\T{\mathcal{T}}
  \def\L{\mathcal{L}}

\def\ol{\overline}

{\left\{\begin{array}{@{}l@{}}}{\end{array}\right.}
\patchcmd{\abstract}{\scshape\abstractname}{\textbf{\abstractname}}{}{}
\makeatletter 
\def\@makefnmark{} 
\makeatother

\numberwithin{equation}{section}

\begin{document}
\title[Power mean curvature flow]{Hypersurfaces with capillary boundary evolving by volume preserving power mean curvature flow}

\author{Carlo Sinestrari}
\address{(Carlo Sinestrari) Dipartimento Di Matematica, Universit\`a degli Studi di Roma "Tor Vergata", Via della Ricerca Scientifica, 00133,  Roma, Italy}

\email{sinestra@mat.uniroma2.it}

\author{Liangjun Weng}

\address{(Liangjun Weng) Dipartimento Di Matematica, Universit\`a degli Studi di Roma "Tor Vergata", Via della Ricerca Scientifica, 00133,  Roma, Italy}
\email{liangjun.weng@uniroma2.it}



\maketitle
\begin{abstract}
In this paper, we introduce a volume- or area-preserving curvature flow for hypersurfaces with capillary boundary in the half-space, with speed given by a positive power of the mean curvature with a non-local averaging term.  We demonstrate that for any convex initial hypersurface with a capillary boundary, the flow exists for all time and smoothly converges to a spherical cap as $t\to +\infty$. 
\end{abstract}


\section{Introduction}

In this paper, we investigate the behavior of a curvature flow of capillary hypersurfaces with a prescribed contact angle condition at the boundary. Curvature flows have been widely investigated in the case of a closed manifold, starting from the classical paper of Huisken \cite{Hui84} on mean curvature flow, which is a time-dependent family of immersions $X: M^n\times[0, T)\to \RR^{n+1}$ of an $n$-dimensional manifold satisfying
\begin{eqnarray*}\label{MC flow}
    \p_tX= -H \nu,
\end{eqnarray*}
where $\nu$ is the outer normal and $H$ is the mean curvature. In \cite{Hui84} it was proved that the mean curvature flow collapses arbitrary initial convex hypersurface to round points in finite time, and this result was the starting point of a fruitful line of research on the formation of singularities which led to relevant geometric applications through the decades.

Later, Schulze \cite{Sch02,Sch05,Sch06} considered a flow of the form $\p_tX= -H^\a \nu$ for a general power $\a>0$ of the mean curvature. He proved that convex hypersurfaces collapse to a point and that the asymptotic shape is round if $\a>1$ and if the principal curvatures of the initial surface are sufficiently pinched. The analysis of a general power has various motivations. As observed in \cite{Sch02}, the $H^\a$-flow is the gradient flow of the area functional of $\S_t:=X(M^n,t)$ with respect to the $L^{\frac{\a+1}{\a}}$-norm, generalizing a well-known property in the $\a=1$ case. Furthermore, a weak formulation of the $H^\a$-flow was later used in \cite{Sch08} to prove isoperimetric inequalities in Riemannian manifolds. We further mention two very recent papers, where it is shown that the power mean curvature flow occurs respectively as the limit of the level set flow of a time-fractional Allen-Cahn equation \cite{DNV24}  and of a suitable minimizing movement evolution of sets \cite{BK24}. In addition to the powers of the mean curvature, a vast literature has been devoted to hypersurface flows where the speed is given by general nonlinear functions of the principal curvatures, see e.g. \cite{And94, A07, AM, AMZ, BCD} and the references therein.

 In \cite{Hui87}, Huisken considered a a \textit{non-local} variant of the mean curvature flow, \begin{eqnarray}\label{huisken flow}
    \p_tX= (h(t)-H) \nu,
\end{eqnarray}
where $h(t):= \frac{1}{|\S_t|}\int_{\S_t} HdA$  is chosen to ensure that the volume enclosed by $\S_t$ remains constant. A remarkable feature of this flow is that the area $|\S_t|$ is monotone decreasing so that the isoperimetric ratio of the enclosed domain improves with the flow. Huisken proved that if $\S_0$ is convex then the flow \eqref{huisken flow} has a solution for all times and converges smoothly to a round sphere as $t\to+\infty$. McCoy \cite{Mc03} obtained the same convergence result by modifying the term $h(t)$ in \eqref{huisken flow} in order to keep the area $|\S_t|$ constant; in this case, the enclosed volume is increasing and the isoperimetric ratio is again improving. Similar results were then obtained for other constrained curvature flows in various settings \cite{And01, CM, CS10,  Mc04, Mc05, Mc17}.

In \cite{S15}, the first author studied the flow analogous to \cite{Hui87} for a general power $\a>0$ of the mean curvature, 
\begin{eqnarray}\label{sine flow}
    \p_t X=(h(t)-H^\a) \nu.
\end{eqnarray}
The main observation was that the monotonicity properties of the constrained flows can be used as a tool to obtain better convergence results than in the standard case (similar ideas had already been used in \cite{And01, CM} in the case $\a=1$). In \cite{S15}  it was proved that \eqref{sine flow} drives any convex initial data to a round sphere as $t \to \infty$ for a general $\a>0$. We remark that, for the standard (local) power mean curvature flow, the smooth convergence to a round profile is only known under strong curvature pinching assumptions on the data if $\a>1$ \cite{Sch05}, and it is even false for some values of $\a<1$, see the introduction in \cite{S15}  for a more detailed discussion. Similar monotonicity-based techniques, coupled with powerful results of convex analysis, have been applied to more general classes of flows in recent years, see \cite{ ALWX,  AW21, BS, BS18}

In  \cite{GL15}, Guan-Li invented a local type of mean curvature flow, which is also volume preserving and area decreasing as the non-local flow in \cite{Hui87}. Its speed function is given by $\p_t X=(n-\<X,\nu\> H)\nu$, and this flow drives star-shaped closed hypersurfaces to a round sphere as $t\to +\infty$. See also \cite{CGLS, GL09, GL21, GLW}, which give many other significant developments for similar flows in various settings afterward. In particular, we point out that the monotonicity properties of these flows, both in the local and nonlocal case, have been applied to obtain geometric inequalities of the Alexandrov-Fenchel type, possibly for more general domains than in the classical setting of convex analysis. 

While all the above references deal with closed hypersurfaces, our interest in this paper lies in studying hypersurfaces \textit{with boundary}
satisfying a Neumann condition. This case has also been studied for a long time, see e.g. \cite{Stahl}, but initially fewer results were known compared to the closed case. In recent years, research has made great strides in constructing flows for hypersurfaces with boundaries and obtaining geometric inequalities.  For instance, works such as \cite{SWX} and \cite{WeX21} have investigated the inverse nonlinear curvature flow for hypersurfaces with free and capillary boundaries in the Euclidean unit ball respectively, and have derived some new Alexandrov-Fenchel type inequalities. For capillary hypersurfaces in the half-space, references such as \cite{MWW} and \cite{HWYZ, WWX22, WWX23} have explored mean curvature flow and inverse curvature flow with capillary boundaries in the half-space, respectively. 
Very recently, the asymptotic behavior of capillary convex hypersurface evolving by the Gauss curvature type flow was analyzed in \cite{MWW23}. 
All these works on flows with boundary concern local type flows. In this paper, instead, we design and analyze a \textit{nonlocal} flow corresponding to \eqref{sine flow} for hypersurfaces with capillary boundaries in the half-space. 
To the best of our knowledge, non-local curvature flows with boundary in the smooth setting have only been studied in dimension $n=1$, such as the area-preserving curve shortening flow with free boundary in $\mathbb{R}^2$ analyzed by M\"{a}der-Baumdicker \cite{Mad15, Mad18}.

Typically, non-local type flows pose different challenges compared to their local counterparts. In particular, some maximum principle arguments can no longer be applied and well-known properties of the local case, such as the avoidance principle, do no longer hold. The capillary setting also poses additional difficulties because it requires analysis of the boundary behavior of geometric quantities. On the other hand, it is remarkable that, by suitably modifying the geometric quantities, we are able to recover in the capillary case the essential properties used for the convergence result in the closed case \cite{S15}, in particular the monotonicity of the isoperimetric ratio adapted to the capillary setting.

To describe our results, we introduce some notation and definitions.
Let $\Sigma$ be a compact hypersurface in the Euclidean closed half-space $ \ol{\RR^{n+1}_+}:=\{x\in \RR^{n+1} \, |\, \<x,E_{n+1}\>\ge 0\}$ ($n\geq 1$), with boundary $\p\Sigma$ lying in $\p \RR^{n+1}_+$, where $E_{n+1}=(0,\ldots, 0,1)$ is the unit inner normal to $\p\RR^{n+1}_+$ in $\ol{\RR^{n+1}_+}$. Such a hypersurface is called a
\textit{capillary hypersurface} if 
$\S$ intersects $\p {{\R}}_+^{n+1}$ at a constant contact angle $\theta\in (0,\pi)$ along $\p\S$, that is,
\begin{eqnarray}\label{contact angle}
\<\nu,E_{n+1}\>=\cos\theta, ~~~ \text{ on } \p \S,
\end{eqnarray}where $\nu$ is the unit outward normal of $\S$. In particular,  if $\S$ meets the boundary orthogonally, i.e. $\theta=\frac \pi 2$,  it is called a hypersurface with \textit{free boundary}. 
Denote  $\widehat \Sigma$  the bounded domain in $\ol{\RR^{n+1}_+}$ enclosed by $\S\subset \ol{\RR^{n+1}_+}$ and by $\p \RR^{n+1}_+$.
The boundary of $\widehat{\S}$  consists of two parts: one is $\Sigma$ and the other, which
will be denoted by $\widehat{\p \Sigma}$,  lies on   $\p \RR^{n+1}_+$. 
$\S$ and $\wh{\p\S}$ have a common boundary $\p\S$. The well-known \textit{capillary area} functional (see for instance the comprehensive books by Finn \cite[Section 1.4]{Finn} and Maggi \cite[Chapter 19]{Maggi}) for $\S$ is defined as  \begin{eqnarray*}\label{wetting}
W_{\theta}(\widehat{\Sigma}):=  |\S|-\cos\theta |\widehat{\p\S}|.\end{eqnarray*} 
Moreover, it is easy to see that (see e.g. \cite[Lemma 2.14]{MWWX} or \cite[Remark 2.1]{PP})
\begin{eqnarray}\label{equivalent}
    W_{\theta}(\wh\S)= \int_\S (1-\cos\theta \<\nu,E_{n+1}\> )dA.
\end{eqnarray}

Let $M$  be a $n$-dimensional smooth compact manifold with boundary $\p M$ and $\S_0:=X_0(M) \subset \ol{\RR^{n+1}_+}$ be a capillary hypersurface with contact angle $\theta\in(0,\pi)$. We consider a family of smooth capillary hypersurfaces $\S_t$ in $\ol{\RR^{n+1}_+}$ starting from $\S_0$, given by the embeddings $X:M \times [0, T) \to \ol{\RR^{n+1}_+}$ and satisfying
\begin{eqnarray}\label{MCF-capillary}
\left\{
\begin{array}{llll}
\p_t X&=&(\phi(t)-H^\a)  \wt\nu, \quad &
\hbox{ in } M\times [0,T),\\
\< \nu ,E_{n+1} \> &=& \cos\theta 
\quad  & \hbox{ on }\partial M\times [0,T),\\
X(\cdot,0)&=& X_0(\cdot) & \hbox{ on } M,
\end{array} \right.
\end{eqnarray} where 
$\wt \nu:=\nu-\cos\theta E_{n+1}$ is the \textit{capillary outward normal} of $\S_t\subset \ol{\RR^{n+1}_+}$. 
It is easy to observe that  the capillary boundary condition $\<\nu,E_{n+1}\>=\cos\theta$  is equivalent to \begin{eqnarray*}    \< \wt \nu,E_{n+1}\>=0,  ~~~\text{ on } \p M.\end{eqnarray*} This ensures that the boundary of $\p \S_t$ evolves inside $\p \ol{\RR^{n+1}_+}$
along the flow \eqref{MCF-capillary}.

In this paper, we choose the non-local term $\phi$ in \eqref{MCF-capillary} to be either
\begin{itemize}
    \item 

\begin{eqnarray}\label{phi-volume pres}
    \phi(t):=\frac{1}{W_{\theta}(\wh{\S_t}) } \int_{\S_t} (1-\cos\theta \<\nu,E_{n+1}\>) H^\a  dA,
\end{eqnarray}
or
\item 
\begin{eqnarray}\label{phi-area pres}
    \phi(t):=\frac{  \int_{\S_t} (1-\cos\theta \<\nu,E_{n+1}\>) H^{\a+1} dA}{\int_{\S_t} (1-\cos\theta \<\nu,E_{n+1}\>) H dA }  .
\end{eqnarray}
\end{itemize}
The motivation for such choices is that in case \eqref{phi-volume pres} the flow \eqref{MCF-capillary} is \textit{volume preserving}, i.e. the volume of $\wh{\S_t}$ remains constant, while under \eqref{phi-area pres} it is \textit{area preserving}, in the sense that the capillary area $W_{\theta}(\wh{\S_t})$ is constant. In both cases, after defining the capillary isoperimetric ratio $\I_\theta(\wh{\S_t})$, see \eqref{cap iso ratio}, we will see in Proposition \ref{cap isoper ratio} that $\I_\theta(\wh{\S_t})$ is non-increasing along the flow, with strict monotonicity unless $\S_t$ is a spherical cap.

The main result of this paper is the following.
 \begin{thm}\label{thm1} Let $\S_0$ be a smooth strictly convex capillary hypersurface in $ \ol{\RR^{n+1}_+}$ with contact angle $\theta\in(0,\frac \pi 2]$. Then the flow \eqref{MCF-capillary}, with $\phi$ defined as in \eqref{phi-volume pres} (resp. \eqref{phi-area pres}) has a unique, smooth solution $\S_t$ which is a strictly convex capillary hypersurface for all $t\in[0,+\infty)$. Moreover, the hypersurfaces $\S_t$  converge smoothly as $t\to +\infty$ to a spherical cap with boundary angle $\theta$ and having the same volume (resp. capillary area) as $\S_0$.  
\end{thm}

In the rest of this article, we assume $\theta\in(0,\frac \pi 2)$, since the case $\theta=\frac \pi 2$ can be reduced to the closed hypersurface counterpart \cite{Hui87,Mc03, S15} by using a simple reflection argument along $\p\RR^{n+1}_+$. We leave out the case $\theta>\frac{\pi}{2}$, where the curvature estimates are not available due to the bad sign of the boundary terms (cf. \eqref{neuman1} and \eqref{neuman2} in the proof of Proposition \ref{convexity pres}). A similar restriction of the contact angle occurs in \cite{HWYZ, WWX22} and \cite{WeX21}.

On the other hand, we believe that some of the techniques introduced in this paper can be used to study a broad spectrum of quermassintegral-preserving non-local flows for capillary hypersurfaces, as explored in \cite{ALWX, AW21, BS18, CS10, Mc05, Mc17}, among others. In particular, by exploiting some Alexandrov-Fenchel inequality in the capillary case \cite{MWW, MWWX}, coupled with curvature estimates, we plan to establish the convergence of the flow by a symmetric polynomial function of the curvature in the capillary setting in a forthcoming work.

Let us summarize our strategy for proving Theorem \ref{thm1}. Initially, we introduce the capillary isoperimetric ratio and we extend to the capillary setting a pinching estimate for the radii of convex sets obtained in \cite{HS15}. More precisely, in Proposition \ref{cap pinching est} we show that the capillary isoperimetric ratio controls the ratio between the capillary outer and inner radius of our hypersurface.  This result is of independent interest and can potentially offer insights into other capillarity-type problems. Together with the monotonicity of the isoperimetric ratio,  this estimate allows to control of the geometry of $\S_t$ as long as the flow exists. In particular, we can adapt Tso's technique \cite{Tso} to derive a uniform upper bound on the mean curvature, replacing the support function with the capillary support function $\bar u$. We then apply the boundary tensor maximum principle from Hu-Wei-Yang-Zhou \cite{HWYZ} to demonstrate the preservation of strict convexity along our flow. The last major step consists of proving a uniform positive bound from below for the mean curvature. This is important because the parabolic operator associated with our problem has a coefficient proportional to $H^{\alpha-1}$, see \eqref{para operator}, thus if $\alpha\neq 1$ the flow becomes either degenerate or singular as $H$ approaches zero. To bound $H$ from below, we use a variant of Tso's technique developed by Bertini-Sinestrari \cite[Section 4]{BS}, combined with an estimate on the position of $\S_t$ based on an Alexandrov reflection argument. In this way, we obtain two-sided curvature bounds which ensure the long-time existence and uniform $C^\infty$-estimates of the solution to the flow. As in \cite[Section 4.2]{S15} or \cite[Section 4]{BS}, we can again use the monotonicity of the isoperimetric ratio to show that the flow must converge to a spherical cap.
 
\

\textbf{The paper is organized as follows}: In Section \ref{sec2}, we recall and derive the relevant evolution equations for our flow and prove the pinching estimate on the capillary outer and inner radius. In Section \ref{sec3}, we show the monotonicity of the capillary isoperimetric ratio and the invariance of the strict convexity, and we obtain the curvature estimates that imply the convergence of the flow.
 
\section{Preliminaries}\label{sec2}

In this section, we collect some preliminaries and recall some evolution equations along the flow \eqref{MCF-capillary}, which will serve as basic tools in the subsequent sections. In addition, we establish a capillary-type pinching estimate (cf. Proposition \ref{cap pinching est}) concerning the capillary outer and inner radius of the convex hypersurface $\S\subset \ol{\RR^{n+1}_+}$.

\subsection{Evolution equations}\ 

Let $\Sigma_t$ denote a family of smooth, embedded hypersurfaces with capillary boundary in $\ol{\R^{n+1}_+}$, defined by the embeddings $X(\cdot,t): M\to \ol{\mathbb{R}^{n+1}_+}$, which evolves according to the general flow
	\begin{eqnarray}\label{flow with normal and tangential}
		\p_t X=\ff\nu+\T,
	\end{eqnarray}  for some speed function $\ff$ and  tangential vector field $\T\in T\Sigma_t$.  Our flow \eqref{MCF-capillary}  can be written in the form \eqref{flow with normal and tangential}  by choosing 
 \begin{eqnarray}
 \ff &=& (1-\cos\theta \<\nu,E_{n+1}\>)  f, \nonumber \\
  \qquad \T &=& -f \cos \theta E_{n+1}^T = f \cos \theta \, ( \<\nu,E_{n+1}\> \nu - E_{n+1} ), \label{T} 
 \end{eqnarray} where $$f:=\phi-H^\a$$
 and $E_{n+1}^T$ is the tangential projection of $E_{n+1}$ onto $T\S_t.$

Let $\mu$ be the unit outward co-normal of $\p\S_t$ in $\S_t$. From the capillary boundary condition in \eqref{MCF-capillary} we know that  
	\begin{eqnarray}\label{co-normal}
		E_{n+1}=\cos\theta \nu-\sin\theta \mu, \quad \text{ on } \p \S_t,
	\end{eqnarray}
	which implies
 \begin{eqnarray}\label{tangential T}
 \T |_{\p \S_t}=f \sin\theta \cos\theta \mu=\ff\cot\theta \mu.
 \end{eqnarray}
 see \cite[Section 2.5]{WWX22} for more details. We also recall the useful property that $\mu$ is a principal direction for the second fundamental form, that is,
 \begin{equation}\label{principal}
\qquad \qquad h(V,\mu)=0, \quad \hbox{ on }\p\S_t, 
\end{equation}
 for any $V$ tangent vector to $\partial \Sigma_t$, see e.g. \cite[Lemma 2.2]{AiSo}.

 We now recall the evolution equations for the induced metric $g_{ij}$, the unit normal $\nu$, the second fundamental form $(h_{ij})$, the Weingarten tensor  $(h^i_j)$ and the mean curvature $H$ of the hypersurfaces $\Sigma_t$ along our flow. We denote by $\n$ and $\Delta$ the Levi-Civita connection and Laplace operator on $\S_t$ with respect to the induced metric, and we set $(h^2)_{ij}:=\sum_k h_{ik}h^k_j$, and $|h|^2:=\sum_{i,j} h_{ij}h^j_i$. Then the following result holds, see \cite[Proposition 2.11]{WeX21} for a detailed proof.
	
	\begin{prop}[\cite{WeX21}]\label{basic evolution eqs}
		Along the flow \eqref{flow with normal and tangential}, there holds 
		\begin{enumerate} 
			\item $\p_t g_{ij}=2\ff h_{ij}+\n_i \T_j+\n_j\T_i$.
			\item $\p_t\nu =-\n \ff+h(e_i,\T)e_i$.	
			\item $\p_t h_{ij}=-\n^2_{ij} \ff +\ff h_{ik}h_{j}^k +\n_\T h_{ij}+h_{j}^k\n_i\T_k+h_{i}^k\n_j \T_k.$
			\item $\p_t h^i_j=-\n^i\n_{j}\ff -\ff h_{j}^kh^{i}_k+\n_\T h^i_j.$
			\item $\p_t H=-\Delta \ff-|h|^2 \ff+ \langle\n  H, \T\rangle $.

		\end{enumerate}
	\end{prop}

The flow \eqref{MCF-capillary} is parabolic if $H>0$ and the local existence of a smooth solution is obtained by standard techniques, see e.g. \cite{Sch05} and \cite[Ch. 18]{ACGL} for the case without boundary, and \cite[Ch. 2]{Wand}, \cite[Section 4]{WWX22} for the extension to the capillary setting. In addition, the solution remains smooth, with bounds on all derivatives, as long as the curvature is bounded and the parabolicity remains strict. It follows that, for any smooth initial data with $H>0$ satisfying the boundary condition, flow \eqref{MCF-capillary} has a unique smooth solution defined in a maximal time interval $[0,T^*)$, where $T^* \leq +\infty$. Moreover, if $T^*$ is finite, then either the second fundamental form of $\S_t$ becomes unbounded as $t \to T^*$, or the infimum of $H$ approaches zero.

For the flow \eqref{MCF-capillary}, we introduce the linearized parabolic operator as
\begin{eqnarray}\label{para operator}
	\mathcal{L}:=\partial_{t}-\a (1-\cos\theta \<\nu,E_{n+1}\>) H^{\a-1} \Delta.
\end{eqnarray}
In the following computations, we use the convention of Einstein summation and the indices appearing after the semi-colon denote the covariant derivatives.
  
For any $z\in \RR^{n+1}$, the \textit{capillary support function} was introduced in \cite[Eq. (2.6)]{WWX23} as 
	\begin{eqnarray}\label{relative support}
		\bar u:=\frac{\<X-z, \nu\>}{1-\cos\theta \<\nu,E_{n+1}\>},
	\end{eqnarray} and is a capillary variant of the support function $u:=\<X-z,\nu\>$. This function will play an important role for us in deriving the curvature estimates of flow \eqref{MCF-capillary} in the next section. We compute the evolution equations for $u$ and $\bar u$. 
 
	\begin{prop}\label{evo of relative u}
		Along the flow \eqref{MCF-capillary}, the  functions $u$ and   $\bar u$ satisfy
  \begin{eqnarray}\label{evo of u}
			\L  u=(1-\cos\theta \<\nu,E_{n+1}\>) \left[ \phi(t)-(1+\a)   H^\a+\a  u  H^{\a-1}|h|^2\right],
		\end{eqnarray}
  and 
		\begin{eqnarray}			\L \bar u & = & \phi -(1+\a) H^\a +\a \bar u H^{\a-1} |h|^2 \nonumber \\
			& & + {2\a H^{\a-1}}  \left\<\n \bar u,\n (1-\cos\theta \<\nu,E_{n+1}\>) \right\>. \label{evo of bar u}
		\end{eqnarray}
On $\p \S_t$, there hold
 \begin{eqnarray}\label{neumann of u}
			\n_{\mu} u=\cot\theta h(\mu,\mu) u,
		\end{eqnarray}  
  and \begin{eqnarray}\label{neumann of bar u}
			\n_{\mu} \bar u=0.
		\end{eqnarray}  
	\end{prop}	
	
	\begin{proof} Since \eqref{neumann of u} and \eqref{neumann of bar u} have been shown in \cite[Proposition 3.1, Proposition 3.3]{MWW}, we only need to prove \eqref{evo of u} and \eqref{evo of bar u}. 
		From the Codazzi formula, 	
  	\begin{eqnarray*}
			\Delta u=\<X-z, \nabla H\>+H-u|h|^{2},
		\end{eqnarray*}then
  \begin{eqnarray*}
      \p_t u&=& \<\p_t X,\nu\>+\<X-z,\p_t \nu\>
      \\&=& f (1-\cos\theta \<\nu,E_{n+1}\>)+\a (1-\cos\theta \<\nu,E_{n+1}\>)  H^{\a-1} \<X-z,\n H\> \\&& -f\< X-z,\n (1-\cos\theta \<\nu,E_{n+1}\>)\> +h((X-z)^T,\T)
      \\&=& \a (1-\cos\theta \<\nu,E_{n+1}\>) H^{\a-1}(\Delta u-H+u|h|^2)+(1-\cos\theta \<\nu,E_{n+1}\>) f \\&& +f \cos\theta h((X-z)^T, E_{n+1}^T) +h((X-z)^T,\T),
  \end{eqnarray*}where $(X-z)^T$ and $E_{n+1}^T$ denote the tangential projection of $X-z$ and $E_{n+1}$ onto $T\S_t$ respectively. The two terms in the last row cancel each other by \eqref{T} and we obtain \eqref{evo of u}.

Next, we derive \eqref{evo of bar u}. By using the Codazzi formula again,  
\begin{eqnarray}\label{delta nu e}
			\De \<\nu,E_{n+1}\> = \<\n H,E_{n+1}\>-\<\nu,E_{n+1}\>|h|^2,
		\end{eqnarray}then 
		\begin{eqnarray*}
			\<\n f,E_{n+1}\>= -\a H^{\a-1}(\Delta \<\nu,E_{n+1}\>+\<\nu,E_{n+1}\> |h|^2).
		\end{eqnarray*}
From Proposition \ref{basic evolution eqs} (2), 		\begin{eqnarray*}
			\p_t \<\nu,E_{n+1}\>&=& -\<\n (f (1-\cos\theta \<\nu,E_{n+1}\>)) ,E_{n+1}\>+h(\T,E_{n+1}^T)
			\\&=& \a (1-\cos\theta \<\nu,E_{n+1}\>) H^{\a-1}(\Delta \<\nu,E_{n+1}\> +|h|^2 \<\nu,E_{n+1}\>) \\&& +
			f \cos \theta \, h(E_{n+1}^T,E_{n+1}^T) +h(\T,E_{n+1}^T)
			\\&=& \a (1-\cos\theta \<\nu,E_{n+1}\>) H^{\a-1}(\Delta \<\nu,E_{n+1}\> +|h|^2 \<\nu,E_{n+1}\>).
		\end{eqnarray*}
		Therefore,
		\begin{eqnarray*}\label{ev of nu e}
			\L \<\nu,E_{n+1}\>= \a (1-\cos\theta \<\nu,E_{n+1}\>) H^{\a-1} \<\nu,E_{n+1}\> |h|^2,
		\end{eqnarray*}and also
  \begin{eqnarray}\label{ev of nu e 1}
      \L (1-\cos\theta \<\nu,E_{n+1}\>)=-\a\cos\theta \<\nu,E_{n+1}\> (1-\cos\theta \<\nu,E_{n+1}\>) H^{\a-1} |h|^2.   \end{eqnarray}
From \eqref{evo of u} and \eqref{ev of nu e 1}, we derive
		\begin{eqnarray*}
			\L \bar u&=& \frac{1}{1-\cos\theta \<\nu,E_{n+1}\>} \L u -\frac{\<X-z,\nu\>}{(1-\cos\theta\<\nu,E_{n+1}\>)^2} \L (1-\cos\theta \<\nu,E_{n+1}\>) \\
			&&+{2\a H^{\a-1}} \<\n \bar u,\n(1-\cos\theta \<\nu,E_{n+1}\>)\> 
			\\&=&
   \phi-(1+\a) H^\a +\a H^{\a-1} u|h|^2 +\cos\theta \<\nu, E_{n+1}\> \a\bar u  H^{\a-1} |h|^2 \\&& + 2\a H^{\a-1} \left\<\n \bar u,\n (1-\cos\theta \<\nu,E_{n+1}\>)  \right\>
   \\&=& \phi -(1+\a) H^\a +\a \bar u H^{\a-1} |h|^2 + {2\a H^{\a-1}}  \left\<\n \bar u,\n (1-\cos\theta \<\nu,E_{n+1}\>)  \right\>.
		\end{eqnarray*}
	 
	\end{proof}	

 Next, we derive the evolution equation for the mean curvature.
 \begin{prop}\label{evo of H}
    Along the flow \eqref{MCF-capillary}, there hold
  \begin{eqnarray}
      \L H &= &\a(\a-1) (1-\cos\theta \<\nu,E_{n+1}\>)  H^{\a-2} |\n H|^2  \notag \\&& +2\a H^{\a-1} \< \n H, \n (1-\cos\theta \<\nu,E_{n+1}\>) \> -(\phi-H^\a) |h|^2, \label{H evol MCF}
  \end{eqnarray} 
    and 
        \begin{eqnarray}\label{neumann of H}
       \n_\mu H =0, \text{ on } \p \S_t. 
\end{eqnarray} \end{prop}
\begin{proof} Using Proposition \ref{basic evolution eqs} (5) and \eqref{delta nu e}, we compute
    \begin{eqnarray*}
        \p_t H&=& -\Delta (f(1-\cos\theta \<\nu,E_{n+1}\>) )-|h|^2 f(1-\cos\theta \<\nu,E_{n+1}\>)  +\< \n H, \T\>
        \\&=& (1-\cos\theta \<\nu,E_{n+1}\>)  (\a H^{\a-1}\Delta H+\a(\a-1) H^{\a-2} |\n H|^2)  \\&& +2\a H^{\a-1} \< \n H,\n (1-\cos\theta \<\nu,E_{n+1}\>) \> 
        \\&&+f \cos\theta (\< \n H,E_{n+1}\>- |h|^2 \<\nu,E_{n+1}\>) \\&& -
        f(1-\cos\theta \<\nu,E_{n+1}\>) |h|^2 +\<\n H, \T\>.
    \end{eqnarray*}Taking into account \eqref{para operator} and \eqref{T}, this implies \eqref{H evol MCF}.

To show \eqref{neumann of H},  let us first take $\{e_\alpha\}_{\alpha=2}^{n}$  an orthonormal frame  of $T(\p\S_t)\subset T\mathbb{R}^n$. Then $\{(e_\alpha)_{\alpha=2}^{n}, e_1:=\mu\}$ forms an orthonormal frame of $T\S_t$. We recall our boundary condition
\begin{equation}\label{boundcond}
\qquad \qquad \langle \nu,E_{n+1}\rangle=\cos\theta \quad  \hbox{ on }\p\S_t.
\end{equation} 
Let us differentiate \eqref{boundcond} with respect to time. Using Proposition \ref{basic evolution eqs} (2), formulas \eqref{co-normal}, \eqref{tangential T}, and \eqref{principal} we obtain along $\p\S_t$
	\begin{eqnarray*}
		0&=&\< \p_t \nu, E_{n+1}\>
		\\&=&\<-\n  \ff+ h(e_i,\T)e_i, \cos\theta \nu-\sin\theta \mu \>
		\\&=&\sin\theta \n_\mu \ff-\cos\theta \ff h(\mu,\mu).
	\end{eqnarray*}
Since $\ff=f(1-\cos\theta \<\nu,E_{n+1}\>)$, this can be rewritten as
	\begin{eqnarray}			\n_\mu (f(1-\cos\theta \<\nu,E_{n+1}\>))=\cot\theta h(\mu,\mu) f(1-\cos\theta \<\nu,E_{n+1}\>).
	\label{first}	\end{eqnarray}
On the other hand, we have by  \eqref{co-normal} and \eqref{principal}
\begin{eqnarray} \nonumber
    \n_\mu (1-\cos\theta \<\nu,E_{n+1}\>)&=& -\cos\theta h(\mu,\mu)\<\mu,E_{n+1}\> \\
&   = & \cot\theta h(\mu,\mu)  (1-\cos\theta \<\nu,E_{n+1}\>). 
\label{second}  
\end{eqnarray}
From \eqref{first} and \eqref{second} we deduce that $\n_\mu f=0$ on $\p \S_t$, that is
\begin{eqnarray}\label{neuman of H}
    \n_\mu H^\a=0 \text{ on } \p \S_t.
\end{eqnarray} 
Since our surfaces satisfy $H>0$, assertion \eqref{neumann of H} follows.
\end{proof}

We point out that, in contrast to the standard power curvature flow $\phi \equiv 0$, the infimum of $H$ is not necessarily increasing and the preservation of the condition $H>0$ is not a direct consequence of the maximum principle. We will see in Proposition \ref{H lower bound prop} that $H$ is indeed bounded from below by a positive constant.

\subsection{Pinching estimate}\ 

Recall that for a convex body $\Omega\subset \RR^{n+1}$, the inner radius and outer radius of $\O$ are defined as 
\begin{eqnarray*}
    \rho_{-}(\O):=\sup\{r>0:  B_r(x_0)\subset \O \text{ for some } x_0\in \RR^{n+1}\},
\end{eqnarray*}and
\begin{eqnarray*}
    \rho_+(\O):=\inf\{r>0:    \O\subset B_r(x_0) \text{ for some } x_0\in \RR^{n+1}\},
\end{eqnarray*}where $B_r(x_0)$ is the ball of radius $r$ centered at $x_0$ in $\RR^{n+1}$. The classical isoperimetric ratio of $\O$ is given by
\begin{eqnarray*}
    \I(\O):= \frac{|\p \O|^{n+1}}{|\O|^n}.
\end{eqnarray*}

For a convex capillary hypersurface $\S\subset \ol{\RR^{n+1}_+}$, we introduce the notion of  the capillary inner radius of $\S$ as
\begin{eqnarray}\label{cap inner rad}
\rho_{-}(\wh{\S},\theta):=\sup\{r>0 ~:~ \widehat{\C_{r,\theta}(x_0)}\subset \widehat\S\text{ for some } x_0\in \p  \RR^{n+1}_+\},\end{eqnarray}and 
 the capillary outer radius of $\S$   as
\begin{eqnarray}\label{cap outer rad}
\rho_+(\wh{\S},\theta):=\inf\{r>0 ~:~\widehat\S\subset \widehat{\C_{r,\theta}(x_0)} \text{ for some } x_0\in  \p \RR^{n+1}_+\},
\end{eqnarray}where 
\begin{eqnarray*}\label{spher cap}
    \C_{r,\theta}(x_0):=\{x\in \ol{\RR^{n+1}_+} ~:~ |x-(x_0-r\cos\theta  E_{n+1})|=r\}
\end{eqnarray*} is the spherical cap centered at $x_0-r\cos\theta E_{n+1}$ with radius $r>0$. For simplicity, when $x_0$ is the origin, we just write $\C_{r,\theta}$ to represent $\C_{r,\theta}(0)$ and $\C_\theta$ to represent $\C_{1,\theta}(0)$.

The \textit{capillary isoperimetric ratio} of $\S$ is defined  as
\begin{eqnarray}\label{cap iso ratio}
    \I_\theta (\wh\S):=\frac{(|\S|-\cos\theta |\widehat{\p\S}|)^{n+1}}{|\widehat\S|^n}= \frac{W_{\theta}(\wh\S)^{n+1}}{|\wh\S|^n}.
\end{eqnarray}
As in the classical case, there holds a capillary isoperimetric inequality (cf. \cite[Theorem 19.21]{Maggi} or \cite[(1.6)]{WWX22}) stating that
\begin{eqnarray}\label{capillary iso ineq}
    \I_\theta(\wh\S)\geq \I_\theta(\wh{\C_\theta}).
\end{eqnarray} Moreover, equality holds in \eqref{capillary iso ineq} if and only if $\S$ is a spherical cap.
 
 In the classical case, it is known that an upper bound on the isoperimetric ratio of a convex body implies a pinching estimate on the outer and inner radii, see \cite[Lemma 4.4]{HS15} or \cite[Proposition 5.1]{And01}. We prove now that an analogous result holds in the capillary case when $0<\theta \leq \frac \pi2$. This is a general property independent of the flow, which has its own interest and may be useful for other capillarity problems. 
 
\begin{prop}\label{cap pinching est}
    Let $\S$ be a convex capillary hypersurface in $\ol{\RR^{n+1}_+}$ and $\theta\in(0,\frac \pi 2]$. 
    If the enclosed domain satisfies $\I_\theta(\wh\S)\leq c_1$ for some positive constant $c_1$, then there exists some  positive constant $c:=c(n,c_1,\theta)$ such that
    \begin{eqnarray}\label{pinching est-cap}
 \rho_+(\wh\S,\theta) \leq c \rho_{-}(\wh\S,\theta).
    \end{eqnarray}
\end{prop}

\begin{proof}
First we observe that the capillary isoperimetric ratio controls from above the classical one.
For any $\theta\in(0,\frac\pi 2]$,  there holds $ (1-\cos\theta)|\S| \leq W_{\theta}(\wh{\S}),$ which yields
    \begin{eqnarray*}
        |\p(\wh\S)|=|\S|+|\wh{\p\S}| \leq 2|\S| \leq \frac{2}{1-\cos\theta} W_{\theta}(\wh{\S}).
    \end{eqnarray*}For a convex capillary domain $\wh\S\subset \ol{\RR^{n+1}_+}$ with $\I_\theta(\wh\S)\leq c_1$, we deduce
    \begin{eqnarray*}
        \I(\wh\S) &=& \frac{|\p(\wh\S)|^{n+1}}{|\wh\S|^n}  \leq \left(\frac 2 {1-\cos\theta}\right)^{n+1} \I_\theta(\wh{\S})  \\& \leq & \left(\frac 2 {1-\cos\theta}\right)^{n+1} c_1.
    \end{eqnarray*} By  \cite[Lemma 4.4]{HS15} applied to  $\wh\S\subset \RR^{n+1}$, we conclude
    \begin{eqnarray}\label{pinching est}
        \rho_+(\wh\S) \leq c_2\rho_{-}(\wh\S),
    \end{eqnarray}for some positive constant $c_2=c_2(n,c_1,\cos\theta)$. This shows that the classical outer and inner radii satisfy a pinching property.
    
To obtain our result, we now have to show that the capillary radii are controlled by the classical ones. We start by estimating from below $\rho_{-}(\wh\S,\theta)$ in terms of $\rho_{-}(\wh \S)$. By definition, $\wh\S$ contains a ball $B_0$, centered at some point $x_0\in \RR^{n+1}$ and of radius $\rho_{-}(\wh \S)$. We  denote by $B_0'$ and $x_0'$  the corresponding projections of $B_0$ and $x_0$ onto $\p\RR^{n+1}_+$.  Observe that  $B_0' \subset  \wh {\p \Sigma}$, by the convexity of $\Sigma$ and the condition $\theta \in (0,\frac \pi2]$.  Let us now consider the spherical cap with angle $\theta$ having the same boundary as $B_0'$, which is $\C_{r_0,\theta}(x_0')\subset \ol{\RR^{n+1}_+}$ with $r_0= \frac{\rho_{-}(\wh\S)}{\sin\theta}$. Then it is easy to see that  $\wh{\C_{r_0,\theta}}\subset \wh{\S}$, see Figure \ref{fig1}. From the definition of capillary inner radius \eqref{cap inner rad}, we deduce \begin{eqnarray}\label{inn-radius-est}
\rho_{-}(\wh\S,\theta) \geq r_0 = \frac{\rho_{-}(\wh\S)}{\sin\theta}.
\end{eqnarray}
\begin{center}	\begin{figure}[H] 	\includegraphics[width=0.8\linewidth]{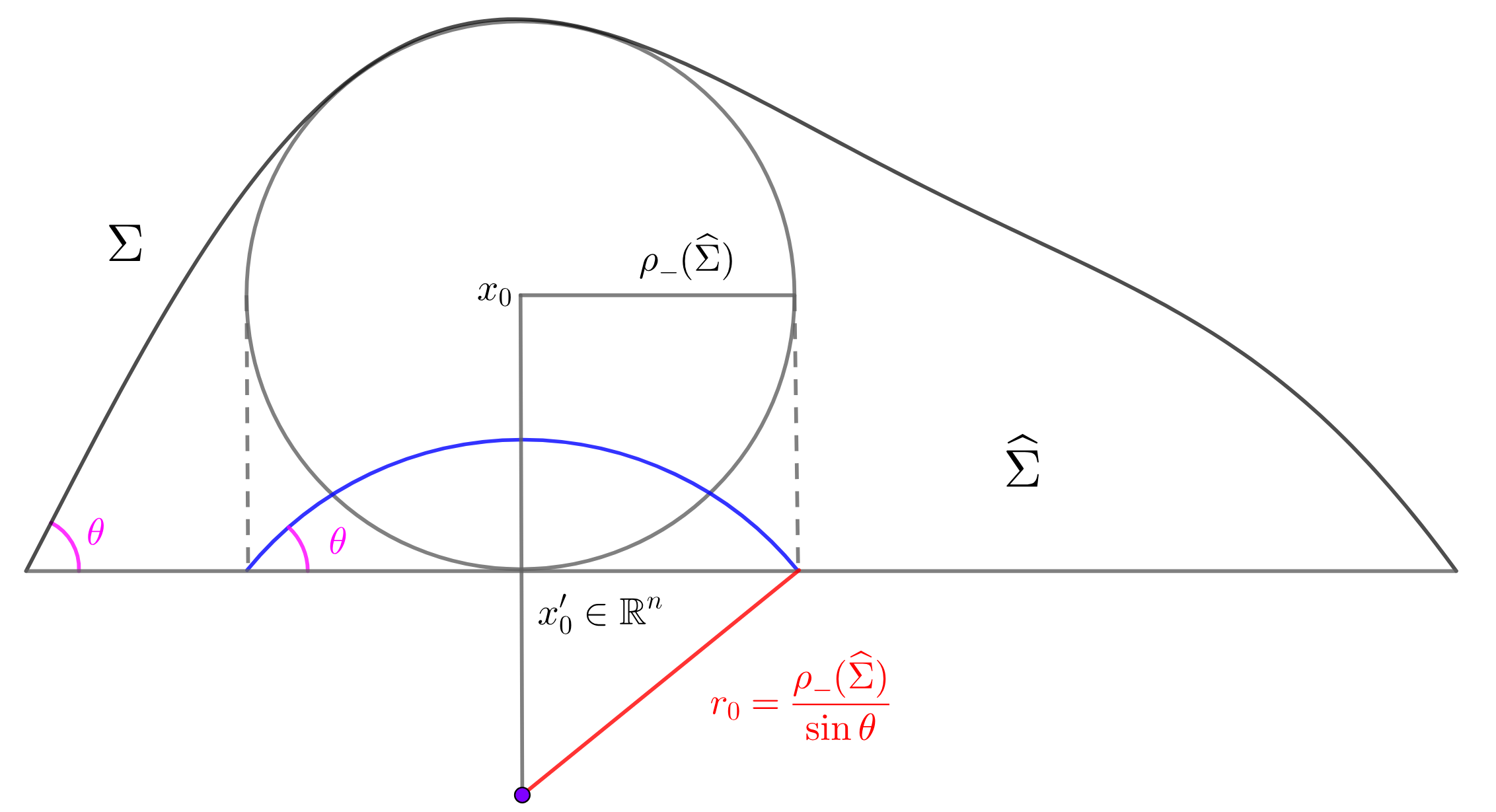}  \caption{Figure 1.}  	
    \label{fig1}  \end{figure}\end{center}

Let us now estimate $\rho_+(\wh\S,\theta)$. Let $B_0$ be a ball of center $x_0$ and radius $\rho_+(\wh\S)$ such that $\wh\S \subset \overline{B_0}$ and let again $x'_0$ be the projection of $x_0$ onto $\p\RR^{n+1}_+$. Since $\wh\S \subset \overline{B_0}$ intersects $\p\RR^{n+1}_+$, the same does $B_0$, which implies $\rho_+(\wh\S) \geq |x_0-x'_0|$. If we now set
$$
R=\frac{2 \rho_+(\wh\S)}{1-\cos \theta},
$$
we see that any $x \in B_0$ satisfies
\begin{eqnarray*}
|x-x_0'+ R \cos \theta E_{n+1}| & \leq & |x-x_0|+|x_0-x_0'| + |R \cos \theta  E_{n+1}| \\
& \leq & \rho_+(\wh\S) + \rho_+(\wh\S) + R \cos \theta = R.
\end{eqnarray*}
This shows that $\overline{B_0} \subset \wh{\C_{R,\theta}(x_0')}$, which in turn implies
\begin{eqnarray}\label{outer rad-est}
    \rho_+(\wh\S,\theta) \leq R = \frac{2}{1-\cos \theta} \rho_+(\wh\S).
\end{eqnarray}

In view of \eqref{inn-radius-est}, \eqref{outer rad-est} and \eqref{pinching est}, the proof of  \eqref{pinching est-cap} is complete.
\end{proof}

\section{A priori estimates}\label{sec3}

In this section, we establish the monotonicity of the capillary isoperimetric ratio $\I_\theta(\wh{\S_t})$ along our flow \eqref{MCF-capillary}, and derive the curvature estimates. This will allow us to prove Theorem \ref{thm1}. Throughout the section, $\S_t$ will be a smooth solution of the flow \eqref{MCF-capillary} with $\theta \in (0,\frac \pi2)$, defined in its maximal time interval $[0,T^*)$.

 \subsection{Monotonicity of capillary isoperimetric ratio}\

 A key  property of our flow \eqref{MCF-capillary} is that the capillary isoperimetric ratio \eqref{cap iso ratio} is monotone non-increasing in time. We remark that this property does not require the convexity of the surfaces and holds for a general $\theta \in (0,\pi)$.
 \begin{prop}\label{cap isoper ratio}
    Along the flow \eqref{MCF-capillary},  the capillary isoperimetric ratio $\I_{\theta}(\wh{\S_t})$  is monotone non-increasing in time $t>0$. Moreover, monotonicity is strict unless $\S_t$ is a spherical cap.
    \end{prop}
\begin{proof}
Let us first look at the case that $\phi(t)$ is chosen as in \eqref{phi-volume pres}. We have
\begin{eqnarray*}
    \frac{ d}{dt} |\widehat{\S_t}|=\int_{\S_t} (\phi-H^\a) (1-\cos\theta \<\nu,E_{n+1}\>) dA=0.
\end{eqnarray*}
Let us denote 
$$q(t):=\frac{  \int_{\S_t} H (1-\cos\theta \<\nu,E_{n+1}\>)dA}{W_{\theta}(\wh{\S_t})}
= \frac{  \int_{\S_t} H (1-\cos\theta \<\nu,E_{n+1}\>)dA} {  \int_{\S_t}(1-\cos\theta \<\nu,E_{n+1}\>)dA},
$$
where we are using \eqref{equivalent}. 
Then the definitions immediately imply the identities
$$
 \phi(t) \int_{\S_t} H (1-\cos\theta \<\nu,E_{n+1}\>)dA=q(t) \int_{\S_t} H^\alpha (1-\cos\theta \<\nu,E_{n+1}\>)dA.
$$
$$
 \int_{\S_t} H (1-\cos\theta \<\nu,E_{n+1}\>)dA - q(t)  \int_{\S_t}(1-\cos\theta \<\nu,E_{n+1}\>)dA =0.
$$
From the first variational formula in \cite[Theorem 2.7]{WWX22}, we deduce, using the two identities above,
\begin{eqnarray*}
    \frac{d}{dt}W_{\theta}(\widehat{\S_t}) &=& \int_{\S_t} (\phi(t)-H^\a) H (1-\cos\theta \<\nu,E_{n+1}\>) dA \\
 & = &    \int_{\S_t} (q(t) H^\a -H^{\a+1}) (1-\cos\theta \<\nu,E_{n+1}\>) dA \\
   & &  +q^\alpha (t)\int_{\S_t} H (1-\cos\theta \<\nu,E_{n+1}\>)dA
   \\ &&- q^{\a+1}(t)  \int_{\S_t}(1-\cos\theta \<\nu,E_{n+1}\>)dA 
 \\&=&-\int_{\S_t} (H-q(t))(H^\a-q^\a(t)) (1-\cos\theta \<\nu,E_{n+1}\>)dA \\&\leq&  0,
\end{eqnarray*}
where the last inequality follows from $\a>0$. Observe in addition that the last inequality is strict unless $H=q(t)$ everywhere om $\S_t$, which means that $\S_t$ is a spherical cap.

If $\phi(t)$ is chosen as in \eqref{phi-area pres}, we define $q(t)$ as above and we obtain this time
$$
\phi(t) \int_{\S_t} (1-\cos\theta \<\nu,E_{n+1}\>)dA=\frac1{q(t)} \int_{\S_t} H^{\alpha+1} (1-\cos\theta \<\nu,E_{n+1}\>)dA.
$$
Using again \cite[Theorem 2.7]{WWX22} we obtain, with computations similar to the previous case,
\begin{eqnarray*}
 \frac{d}{dt}W_{\theta}(\widehat{\S_t}) =  0,
\end{eqnarray*}
and 
\begin{eqnarray*}
    \frac{d}{dt} |\wh{\S_t}|  &=& \int_{\S_t} (\phi-H^\a)  (1-\cos\theta \<\nu,E_{n+1}\>) dA
    \\&=&\frac{1}{q(t)} \int_{\S_t} (H-q(t))(H^\a-q^\a(t)) (1-\cos\theta \<\nu,E_{n+1}\>)dA \\&\geq&  0.
\end{eqnarray*}

In both cases, we conclude that $\I_\theta (\wh{\S_t})$ is monotone non-increasing in $t$.
\end{proof}
 
 \subsection{Curvature estimates}\ 

Now we prove that the convexity is preserved along our flow \eqref{MCF-capillary}.
\begin{prop}\label{convexity pres}
    If $\S_0\subset \ol{\RR^{n+1}_+}$ is a strictly convex capillary hypersurface and $\theta\in(0,\frac \pi 2)$, then the solution $\S_t$ of flow \eqref{MCF-capillary} is strictly convex for all $t\in[0,T^*)$.  
\end{prop}

To demonstrate Proposition \ref{convexity pres}, we begin by recalling the following tensor maximum principle from \cite[Theorem 1.2]{HWYZ},  which is a refinement of the result by Stahl \cite[Theorem 3.3]{Stahl}. Both can be regarded as the boundary counterpart of the tensor maximum principle on a compact manifold established by Hamilton \cite[Theorem 9.1]{Ham} and further refinement by Andrews \cite[Theorem 3.2]{A07}.

\begin{lem}[\cite{HWYZ}]\label{tensr MP thm} Let $N$ be a compact manifold with boundary and $g$ a time-dependent metric. Let  $S(t)$ be a smooth symmetric $(2,0)$-tensor field satisfying
	\begin{eqnarray*} 
	\frac{\partial }{\partial t}S =A^{ij}\nabla_i\nabla_j S+ B^k\nabla_k S+\beta, 
	\end{eqnarray*}where $(A^{ij})>0$ and $B^k$ are smooth. Suppose  $\beta(t)=\beta(S(t),g(t),t)$ is a symmetric $(2,0)$ tensor which satisfies 
	\begin{eqnarray}\label{tensor MP 1}
 \beta(\xi,\xi)+\sup\limits_{\Lambda}  2A^{kl}(2\Lambda^p_k \nabla_l S_{ip} \xi^i-\Lambda^p_k\Lambda^q_l S_{pq}) \geq 0, ~~  	\text{ in }  N\times (0,T],
\end{eqnarray}  and 
\begin{eqnarray}\label{tensor MP 2}
(\n_\mu S_{ij})\xi^i\xi^j \geq 0, ~~ \text{ on } \p N\times (0,T] ,  
	\end{eqnarray}for $\xi$ being a null eigenvector of $S$, i.e. $S_{ij}\xi^j=0$ and $\mu$ being the outward unit normal vector to $\partial N$. If $S(0)> 0$ on $N$, then  $S(t)>0$ on $N$ for all $0< t\leq T$.
\end{lem}

\begin{proof}[\textbf{Proof of Proposition \ref{convexity pres}}]
From Proposition \ref{basic evolution eqs} (3), we have
\begin{eqnarray*}
    \p_t h_{ij} &=& -\n_i\n_j (f(1-\cos\theta \<\nu,E_{n+1}\>) )+f(1-\cos\theta \<\nu,E_{n+1}\>) (h^2)_{ij} \\&& +\n_\T h_{ij}+h_j^k \n_i \T_k +h_i^k \n_j \T_k
    \\&=& (1-\cos\theta \<\nu,E_{n+1}\>) \n_i\n_j H^\a +\n_i  H^\a \n_j(1-\cos\theta \<\nu,E_{n+1}\>) \\&& +\n_j H^\a \n_i(1-\cos\theta \<\nu,E_{n+1}\>)-(\phi-H^\a) \n_{ij}(1-\cos\theta \<\nu,E_{n+1}\>)
    \\&& +(1-\cos\theta \<\nu,E_{n+1}\>)(\phi-H^\a) (h^2)_{ij}+ \n_\T h_{ij}+h_j^k \n_i \T_k +h_i^k \n_j \T_k
    \\&=& \a (1-\cos\theta \<\nu,E_{n+1}\>) (H^{\a-1}H_{;ij}+(\a-1)H^{\a-2}H_{;i}H_{;j}) 
    \\&& -\a \cos\theta H^{\a-1} ( H_{;i} h_{jk}\<e_k,E_{n+1}\> +H_{;j} h_{ik} \<e_k,E_{n+1}\>) 
    \\&&  +(\phi-H^\a) \cos\theta (\<\n h_{ij},E_{n+1}\>-(h^2)_{ij}\<\nu,E_{n+1}\>) 
     \\&& +(1-\cos\theta \<\nu,E_{n+1}\>) (\phi-H^\a) (h^2)_{ij}+ \n_\T h_{ij}+h_j^k \n_i \T_k +h_i^k \n_j \T_k.
     \end{eqnarray*}
We recall the Simons-type identity, see e.g. \cite[Lemma 2.1]{Hui84}.
\begin{eqnarray*}
H_{ ;ij}=\Delta h_{ij }-H(h^2)_{ij}+|h|^2h_{ij}.
\end{eqnarray*}
Then we derive
\begin{eqnarray*}
\p_t h_{ij}&=&  \a(1-\cos\theta \<\nu,E_{n+1}\>) H^{\a-1} \Delta h_{ij} +(\phi-H^\a)\cos\theta \<\n h_{ij},E_{n+1}\>
\\ & & +\n_\T h_{ij} + \a(\a-1)(1-\cos\theta \<\nu,E_{n+1}\>) H^{\a-2} H_{;i}H_{;j}
\\ & & -\a \cos\theta H^{\a-1} ( H_{;i} h_{jk}\<e_k,E_{n+1}\>   +H_{;j} h_{ik} \<e_k,E_{n+1}\>) 
\\&& +(1-2\cos\theta \<\nu,E_{n+1}\>) (\phi-H^\a) (h^2)_{ij} +h_j^k \n_i \T_k +h_i^k \n_j \T_k \\&&  
 + \a(1-\cos\theta \<\nu,E_{n+1}\>)H^{\a-1}( |h|^2h_{ij} -H (h^2)_{ij}).
\end{eqnarray*}

To apply Lemma \ref{tensr MP thm}, we first check the boundary condition \eqref{tensor MP 2}. Assume that $h_{ij}\geq 0$ and $h_{ij}\xi^j=0$ holds at the point $p_0\in \p M$. We choose an orthonormal frame $\{e_i\}_{i=1}^n$ around $p_0$ such that $e_1=\mu$ is the conormal of $\p M$ in $M$. Recalling \eqref{principal}, we can choose our frame in such a way that one of the following holds: either $\xi=e_1$ or $\xi=e_\a\in T(\p \S_t)$ for some $2\leq \a \leq n$.
\begin{enumerate}
    \item If $\xi=\mu$, we know $h_{\a\a}\geq h_{11}$ for all $2\leq \a \leq n$. From \eqref{neuman of H}, we know
    \begin{eqnarray*}
        \n_\mu h_{11}=-\sum_{\a=2}^n \n_\mu h_{\a\a}.
    \end{eqnarray*}
    By applying \cite[Proposition 2.4]{WWX22} and using $\theta\in(0,\frac \pi2 )$, we obtain
    \begin{eqnarray}\label{neuman1}
        (\n_\mu h_{ij})\xi^i\xi^j=\n_\mu h_{11}=-\sum_{\a=2}^n \cot\theta \, h_{\a\a} (h_{11}-h_{\a\a}) \geq 0.
    \end{eqnarray}

    \item If $\xi=e_\a\in T(\p\S_t)$ for some $2\leq \a \leq n$, then $h_{11}\geq h_{\a\a}$, we use again \cite[Proposition 2.4]{WWX22} and $\theta\in(0,\frac \pi 2)$ to deduce that 
    \begin{eqnarray}\label{neuman2}
        (\n_\mu h_{ij})\xi^i\xi^j =\n_\mu h_{\a\a} =\cot\theta \, h_{\a\a} (h_{11}-h_{\a\a}) \geq 0.
    \end{eqnarray}
\end{enumerate}

Next, we check condition \eqref{tensor MP 1}. We assume that $\xi$ is a null eigenvector of $(h_{ij})\geq 0$, i.e. $h_{ij}\xi^j=0$ at the point $(p_0,t_0)$ with $p_0\in M\setminus \p M$. We choose an orthonormal frame $\{e_i\}_{i=1}^n$ around $p_0$ such that $(h_{ij})$ is diagonal and $\xi=e_1$,  then $h_{11}=0$ and $(h^2)_{11}=\sum\limits_{k=1}^n h_1^k h_k^1=0$. Then many terms in the evolution equation for $h_{ij}$ vanish and we find that condition \eqref{tensor MP 1} in our case takes the following form
\begin{eqnarray}\label{Q}
    Q&:=&\a(\a-1) (1-\cos\theta \<\nu,E_{n+1}\>)  H^{\a-2} (\n_1H)^2  \notag \\&& +2\a(1-\cos\theta \<\nu,E_{n+1}\>)  H^{\a-1} \sup_\Lambda (2\La_k^l \n_k h_{l1}-\La_k^i\La_k^j h_{ij})\\& \geq & 0. \notag 
\end{eqnarray}
Here we no longer distinguish between upper and lower indices since we are working in an orthonormal frame.

If $\a\geq 1$, it is easy to see that $Q\geq 0$, by choosing $\La\equiv 0$ in \eqref{Q}.
Hence we only need to show \eqref{Q} for  $\a\in(0,1)$.
 By continuity, we can assume that $(h_{ij})$ has all eigenvalues $\{\k_i\}_{i=1}^n$ distinct and in increasing order at $(p_0,t_0)$ as $\k_1<\k_2<\cdots< \k_n$. From the minimality of $\k_1=h_{11}$ (cf. \cite[Lemma 4.1]{HWYZ}, we have that $0=\n_kh_{11}$ for all $1\leq k\leq n$.  By choosing $\La_k^l=\frac{\n_1 h_{kl}}{h_{ll}}$ for  $2\leq l\leq n, 1\leq k\leq n$ and $\Lambda^1_k\equiv 0$ for $1\leq k\leq n$, we have

\begin{eqnarray*}
&& 2\La_k^l \n_1 h^k_l-\La_k^i\La^{jk} h_{ij} \\&= &
\sum_{l=2}^n  \sum_{k=1}^n \left(2 \frac{\n_1 h_{kl}}{h_{ll}} \n_1h_{kl}-\left(\frac{\n_1 h_{kl}}{h_{ll}} \right)^2h_{ll} \right)
\\&=&\sum_{l=2}^n   \sum_{k=1}^n  \frac{(\n_1h_{kl})^2}{h_{ll}} 
\geq \sum_{l=2}^n    \frac{(\n_1h_{ll})^2}{h_{ll}} \geq  \frac{(\n_1 H)^2}{H},
\end{eqnarray*}where in the last step we have used the Cauchy-Schwarz inequality
\begin{eqnarray*}
&& |\n_1 H | = \left| \sum_{l=2}^n 
 \frac{ \n_1h_{ll}} {\sqrt{h_{ll}}} \, \sqrt{h_{ll}}  \, \right| \leq
\left( \sum_{l=2}^n    \frac{(\n_1h_{ll})^2}{h_{ll}} \right)^{1/2} H^{1/2} .
\end{eqnarray*}
Substituting into \eqref{Q},  it implies
\begin{eqnarray*}
    Q &\geq & \a (1-\cos\theta \<\nu,E_{n+1}\>) H^{\a-2} \left[ (\a -1) (\n_1 H)^2 +  2{(\n_1 H)^2}\right]
   \\&=&  \a(\a+1) (1-\cos\theta \<\nu,E_{n+1}\>) H^{\a-2} (\n_1 H)^2 \geq 0,
\end{eqnarray*}
that is, \eqref{Q} holds also in this case.
In summary, the theorem follows by using Lemma \ref{tensr MP thm}.
\end{proof}

\begin{prop}\label{bound outer-inner radius}
    Along the flow \eqref{MCF-capillary}, if $\theta\in(0,\frac {\pi} 2)$, the enclosed volume $|\wh{\S_t}|$, the capillary area $W_{\theta}(\wh{\S_t})$, the capillary outer radius $\rho_+(\wh{\S_t},\theta)$ and the capillary inner radius $\rho_-(\wh{\S_t},\theta)$ are uniformly bounded from above and below by positive constants. In particular, there exist positive constants $R_0$ and $r_0$, depending only on the initial datum, such that
    \begin{eqnarray}\label{inn-out radius}
        0<r_0\leq \rho_-(\wh{\S_t},\theta)\leq \rho_+(\wh{\S_t},\theta) \leq R_0.
    \end{eqnarray}
   
\end{prop}

\begin{proof}
From Proposition \ref{cap isoper ratio} and \eqref{capillary iso ineq}, we have
\begin{eqnarray}\label{ratio bdds}
 \I_\theta(\wh{\C_\theta})\leq \I_\theta(\wh{\S_t}) \leq \I_\theta(\wh{\S_0}),
\end{eqnarray}for all $t>0$.

If $\phi(t)$ is chosen as  in \eqref{phi-volume pres}, we have  $$|\wh{\S_t}|=|\wh{\S_0}|,$$ for all $t\geq 0$, then the uniform bounds on the capillary area $W_{\theta}(\wh{\S_t})$ follow from \eqref{ratio bdds}. In addition, from  \eqref{ratio bdds} and Proposition \ref{cap pinching est} we deduce
\begin{eqnarray*}
    |\wh{\S_t}| &\leq & |  \wh{ \C_{\rho_+({\wh{\S_t}},\theta),\theta}}| =(\rho_+({\wh{\S_t},\theta} )) ^{n+1} |\wh{\C_\theta}| \\& \leq  & (c\rho_{-}({\wh{\S_t},\theta}) )^{n+1} |\wh{\C_\theta}| \leq c^{n+1}|\wh{\S_t}|.
\end{eqnarray*} 
Since $|\wh{\S_t}|$ is constant in time, this implies the uniform bounds on $\rho_{\pm}({\wh{\S_t}},\theta)$.

In the case that $\phi(t)$ is chosen as in \eqref{phi-area pres}, we know that
\begin{eqnarray*}
    W_{\theta}(\wh{\S_t})=W_{\theta}(\wh{\S_0}), 
\end{eqnarray*} for all $t\geq 0$. From  Proposition \ref{cap isoper ratio} and \eqref{ratio bdds}, we have
$$
    |\wh{\S_0}| \leq  |\wh{\S_t}| =\I_\theta(\wh{\S_t})^{-\frac 1 n} W_{\theta}(\wh{\S_t})^{\frac {n+1} n}
    \leq   \I_\theta(\wh{\C_\theta})^{-\frac 1 n} W_{\theta}(\wh{\S_0})^{\frac {n+1} n},
$$
showing that $|\wh{\S_t}|$ is uniformly bounded. The remaining quantities can be estimated similarly as above.
\end{proof}

Proposition \ref{bound outer-inner radius} implies the existence of a capillary spherical cap with radius $r_0$ contained within $\wh{\S_t}$ for any $t\in[0, T^*)$. The center of the cap, however, may in principle depend on $t$. 
To use Tso's method \cite{Tso}  for obtaining an upper bound on the mean curvature, it is necessary to establish the existence of a cap with a fixed center that remains inside $\wh{\S_t}$ over a suitable uniform time interval. This is done in the next lemma, by adapting a technique from \cite{And01,Mc04}.

 \begin{lem}\label{lower bdd bar u}
 Let $\S_t$ be the solution of \eqref{MCF-capillary}. For any 
 $s \geq 0$, let $z_*\in \wh{\S_s} \cap \p \RR^{n+1}_+$ be a point such that $\wh{\C_{r_0,\theta}(z_*)} \subset \wh{\S_s} $, with $r_0$ as in \eqref{inn-out radius}. Then there exists some $\tau_0>0$, which only depends on the initial datum, such that $\S_t$ satisfies
     \begin{eqnarray}\label{inner cap enclose}
        \wh{\C_{\frac{r_0}2,\theta}(z_*)} \subset   \wh{\S_t} 
     \end{eqnarray}for all $t\in [s,\min\{s+\tau_0,T^*\})$.
     
\end{lem}
\begin{proof}
The idea of the proof is to compare $\S_t$ with the capillary spherical cap centered at $z_*-r(t)\cos\theta E_{n+1}$ with shrinking radius $r(t)$, where $r(t)$ satisfies
 \begin{eqnarray}\label{defr-0}
     r'(t)= -\left(\frac{n}{r(t)}\right)^\a,
 \end{eqnarray} with $r(s)=\frac{99}{100}r_0$. 
We define \begin{eqnarray*}
     F(t):= \frac 1 2 \left(\min_M |X(\cdot,t)-(z_*-r(t)\cos\theta E_{n+1})|^2-r^2(t)\right),
 \end{eqnarray*}
for any $t \geq s$ such that $\S_t$ and $r(t)$ exist.
 It is easy to see that $F(s) > 0$. The conclusion \eqref{inner cap enclose} will follow easily once we show that $F(t)>0$ for  $t>s$.
  
 Suppose $t_0>s$ be the first time such that $F(t)=0$, and let $p_0 \in M$  be a minimum value point of 
 $$
 |X(p_0,t_0)-(z_*-r(t_0)\cos\theta E_{n+1})|=\min_M |X(\cdot,t_0)-(z_*-r(t_0)\cos\theta E_{n+1})|=r(t_0).
 $$ Next we divide the proof into two cases: $p_0\in \mathring M$ and $p_0\in \p M$. 
 \begin{enumerate}
     \item 
$p_0 \in \mathring M$. Then at $p_0$ we have
 \begin{eqnarray}\label{normalpzero 0}
     \nu=\frac{X-(z_*-r(t_0)\cos\theta E_{n+1})}{|X-(z_*-r(t_0) \cos\theta E_{n+1})|}=\frac{X-(z_*-r(t_0) \cos\theta E_{n+1})}{r(t_0)},
 \end{eqnarray}
 which also implies
  \begin{eqnarray}\label{nutildepzero 0}
  \wt \nu = \nu-\cos\theta E_{n+1}=\frac{X(p_0,t_0)-z_*}{r(t_0)}.
   \end{eqnarray}
Moreover, the spherical cap $\C_{r(t_0),\theta}(z_*-r(t_0)\cos\theta E_{n+1})$ is tangent to $\S_{t_0}$ at $X(p_0,t_0)$ from the interior. It follows that the principal curvatures $\k_i$ of $\S_{t_0}$ satisfy 
\begin{eqnarray*}
    \k_i(p_0,t_0) \leq  \frac  1 {r(t_0)}, ~~\forall ~1\leq i \leq n.
\end{eqnarray*} Then
\begin{eqnarray}\label{H compare}
    H(p_0,t_0)\leq \frac  n {r(t_0)}.
\end{eqnarray}

\item 
$p_0 \in \partial M$. By \eqref{principal}, if we choose $\{e_\a\}_{\a=1}^{n-1}$ to be an orthonormal frame on $T\p\S_{t_0}\subset T\RR^n$ and set $e_n:=\mu$, then $\{e_i\}_{i=1}^{n}$ is an orthonormal frame on $T\S_{t_0}$.
From \cite[Proposition 2.4 (2)]{WWX22}, we know  
\begin{eqnarray}\label{kappa-hat}
    \k_\a=\sin \theta \, \wh{\k}_\a,
\end{eqnarray} where $\{\wh\k_\a\}_{\a=1}^{n-1}$ are the principal curvatures of $\p\S_t$ in $\p\RR^{n+1}_+=\RR^n$.

By definition of $p_0$, we have that
 \begin{eqnarray*}
\lefteqn{|X(p_0,t_0)-(z_*-r(t_0)\cos\theta E_{n+1})| ^2} \\
 & = & \min_M |X(\cdot,t_0)-(z_*-r(t_0)\cos\theta E_{n+1})|^2 \\
&=&  \min_{\partial M} |X(\cdot,t_0)-(z_*-r(t_0)\cos\theta E_{n+1})|^2.
 \end{eqnarray*}
On the other hand, for all $p \in \partial M$ we have
$$ \qquad \qquad |X(p,t_0)-z_*+r(t_0)\cos\theta E_{n+1}|^2=|X(p,t_0)-z_*|^2+r^2(t_0)\cos^2\theta.$$
It follows that $p_0$ attains the minimal value of $\min\limits_{\p M}|X(\cdot, t_0)-z_*|^2$. Then we have that the horizontal component of $\nu$ is parallel to $X-z_*$ at $p_0$. In addition, if we regard $\p\S_{t_0}$ as a subset of $\p\RR^{n+1}_+=\RR^n$, we have that 
 the $n$-dimensional ball $B_{|X(p_0,t_0)-z_*|}^n(z_*)\subset \p\RR^{n+1}_+$
 is tangent to $\p\S_{t_0}$ from the inside at $X(p_0,t_0)$. This yields
 \begin{eqnarray}\label{wh kappa}
     \wh{\k}_\a \leq \frac 1 {|X(p_0,t_0)-z_*|}, \quad \forall ~ 1\leq \a \leq n-1,
 \end{eqnarray}and 
\begin{eqnarray}\label{tilde nu}
    \wt \nu=\frac{|\wt \nu|}{|X-z_*|} (X-z_*)= \frac{| \nu - \cos\theta E_{n+1}|}{r(t_0) \sin \theta}(X-z_*) = \frac{X-z_*}{r(t_0)},
\end{eqnarray}
where the last equality used \eqref{contact angle}. That is, \eqref{nutildepzero 0} holds in this case, and also \eqref{normalpzero 0}. Furthermore,  from \eqref{tilde nu}, we have
\begin{eqnarray*}
    \<\mu, X-z_*+r\cos\theta E_{n+1}\>=r(t_0) \<\mu,\nu\>=0,
\end{eqnarray*}that is, $\mu$ is tangent to the spherical cap centred at $X-z_*+r\cos\theta E_{n+1}$ with radius $r(t_0)$, and this implies that
\begin{eqnarray}\label{h nn}
    h(\mu,\mu) \leq \frac 1{r(t_0)}.
\end{eqnarray}In view of \eqref{wh kappa}, \eqref{kappa-hat} and $|X(p_0,t_0)-z_*|=r(t_0)\sin\theta$, we have
\begin{eqnarray*}
    \k_\a \leq \frac 1 {r(t_0)}, ~~ \forall~ 1\leq \a\leq n-1.
\end{eqnarray*}Together with \eqref{h nn}, we conclude that \eqref{H compare} holds also in this case.

 \end{enumerate}
In the following we use standard properties about the derivative of the minimum of a family of smooth functions; we also assume for simplicity that $F$ is differentiable at $t=t_0$, since the argument can be extended to the general case by using Dini derivatives.

We have seen that, in both cases, we have \eqref{normalpzero 0}, \eqref{nutildepzero 0} and \eqref{H compare}.
Taking also into account \eqref{defr-0}, we compute
 \begin{eqnarray*}
  \left.  \frac{d}{dt} F(t) \right|_{t=t_0} &=&\< X-z_*+r\cos\theta E_{n+1}, \p_tX+r'(t_0) \cos\theta E_{n+1}\>-rr'(t_0)
\\ &= & \left\<r\nu, (\phi-H^\a) \frac{X-z_*}{r}+r'(t_0)  \left(\nu-\frac{X-z_*}{r}\right) \right \>-rr'(t_0)
     \\&=&  (\phi-H^\a) \<X-z_*,\nu\>-r'(t_0) \<X-z_*,\nu\> 
     \\&>&  \<X-z_*,\nu\>  \left[ -H^\a+\left(\frac n {r(t_0)}\right)^\a\right]\geq 0,
 \end{eqnarray*}
which contradicts the definition of $t_0$. This shows that $F(t)$ cannot vanish and therefore remains positive for $t>s$. 
Now it suffices to choose $\tau_0>0$ such that $r(t) \geq \frac {r_0} 2$ for all $t\in [s, s+\tau_0)$. Note that from \eqref{defr-0}, we know $\tau_0$ only depends on $r_0$, hence only on the initial datum. Then we have $F(t)>0$ for all $t\in [s,\min\{s+\tau_0, T^*\})$. This completes the proof.

\end{proof}

We can now prove a uniform upper bound for the curvature of our flow.

\begin{prop}\label{H upper bdd prop}
If $\S_t$ be a solution of the flow \eqref{MCF-capillary}, there holds 
    \begin{eqnarray}\label{H upper bdd}
        H \leq C \mbox{ on }\S_t,~~~\forall ~ t\in [0, T^*),
    \end{eqnarray} where $C$ is a positive constant, only depending on the initial datum.
\end{prop}
\begin{proof} 

For any given $s\in[0, T^*)$,  let $z_*$ and $\tau_0$ be chosen as in Lemma \ref{lower bdd bar u}. Then \eqref{inner cap enclose} and the convexity of $\S_t$ imply 
$$
\< X - z_*+\frac{r_0}{2} \cos \theta E_{n+1} , \nu \> \geq \frac{r_0}{2}, 
$$
showing that the capillary support function satisfies 
\begin{eqnarray*}\label{bar u lower bdd}
    \bar u:=\frac{\<X-z_*,\nu\>}{1-\cos\theta \<\nu, E_{n+1}\>} \geq \frac {r_0} 2>0,~~ ~ \forall ~t\in [s,\min\{s+\tau_0, T^*\}).
\end{eqnarray*}
We denote $\ve:=\frac {r_0} 4$, so that $\bar u-\ve \geq \ve>0$,  and introduce the function  
\begin{eqnarray}\label{test-Phi}
    \Phi:=\frac{H^\a}{ \bar u-\varepsilon}.
\end{eqnarray}
From Proposition \ref{evo of H}, we deduce that
 \begin{eqnarray}\label{evo of H alpha}
        \L H^\a=-\a H^{\a-1} (\phi-H^\a) |h|^2 +2\a H^{\a-1} \< \n H^\a,\n (1-\cos\theta \<\nu,E_{n+1}\>)\>.
    \end{eqnarray}
 Combining \eqref{evo of bar u} and \eqref{evo of H alpha},  we obtain
 \begin{eqnarray*}
    \L \Phi&= &  \frac{\L H^\a}{\bar u-\ve} -\frac{H^\a \L \bar u}{(\bar u-\ve)^2}+\frac{2\a  H^{\a-1}}{\bar u-\ve} (1-\cos\theta \<\nu,E_{n+1}\>) \<\n \Phi,\n \bar u\>
    \\&=& \frac{1}{\bar u-\ve} (-\a H^{\a-1}f |h|^2 +2\a H^\a \< \n H^\a,\n (1-\cos\theta \<\nu,E_{n+1}\>)\>) 
    \\&& -\frac{\Phi}{\bar u-\ve}  \left[ \phi -(1+\a) H^\a +\a \bar u H^{\a-1} |h|^2 \right ] 
   \\&&  -\frac{2\a H^{\a-1}\Phi}{\bar u-\ve}    \left\<\n \bar u,\n (1-\cos\theta \<\nu,E_{n+1}\>) \right\> 
     \\&&+\frac{2\a   H^{\a-1}}{\bar u-\ve}  (1-\cos\theta \<\nu,E_{n+1}\>) \<\n \Phi,\n \bar u\>
    \\&=& -(\a H^{\a-1} |h|^2+\Phi) \frac{\phi}{\bar u-\ve}-\frac{\ve \a H^{2\a-1} |h|^2}{(\bar u-\ve)^2} +(1+\a) \Phi^2
     \\&&+2\a H^{\a-1}\<\n \Phi,\n (1-\cos\theta \<\nu,E_{n+1}\>) \>+\frac{2\a \v H^{\a-1}}{\bar u-\ve}  \<\n \Phi,\n \bar u\>
     \\&\leq & -\frac{\ve \a H^{2\a-1} |h|^2}{(\bar u-\ve)^2} +(1+\a) \Phi^2, ~~~\text{ mod } \n \Phi,
\end{eqnarray*}
where we have used that $\phi>0$ by the convexity of $\S_t$. Using also $$|h|^2 \geq \frac{H^2}{n},$$ we conclude that
\begin{eqnarray}\label{L Phi}
    \L \Phi \leq - \frac{\a}{n}\ve^{1+\frac 1 \a}  \Phi^{2+\frac 1 \a}  +(1+\a)\Phi^2 
 ~~~\text{ mod } \n \Phi,
\end{eqnarray}holds on $t\in [s, \min\{s+\tau_0,T^*\})$.
    
On $\p M$, by \eqref{neumann of H} and \eqref{neumann of bar u} we have
\begin{eqnarray}\label{phi neumann}
    \n_\mu \Phi=\frac{\n_\mu H^\a}{\bar u-\varepsilon}-\frac{H^\a \n_\mu \bar u}{(\bar u-\varepsilon)^2}
= 0.
\end{eqnarray}
By the Hopf boundary point lemma, this shows that $\Phi$ attains its maximum value either at $t=0$ or at some interior point, say $p_0\in \mathring M$. We define  $$\wt{\Phi}(t):=\sup_{M} \Phi(p,t).$$ From \eqref{L Phi}, we know $\wt \Phi$ satisfies
\begin{eqnarray}\label{ODE Phi}
    \frac{d}{dt} \wt{\Phi}(t) \leq  \wt \Phi^2 \left(1+\a-\frac{\a}{n}\varepsilon^{1+\frac 1 \a} \wt\Phi^{\frac 1 \a}\right), ~~~\text{ a.e. } t
  \in  [s,\max\{s+\tau_0,T^*\}).
\end{eqnarray}
In the case $s=0$, we argue as follows. From \eqref{ODE Phi} we see that, if $\wt\Phi>  \left(\frac{2n(1+\a)}{\a \ve^{1+\frac 1 \a}}\right)^\a$ at some time, then it does not increase. This implies
\begin{eqnarray*}
    \Phi (\cdot, t) \leq \max\left\{ \max_M \Phi(\cdot, 0), \left(\frac{2n(1+\a)}{\a \ve^{1+\frac 1 \a}}\right)^\a \right\},
\end{eqnarray*}for $t\in \left[0,\min \left\{ \tau_0 , T^*\right\}\right)$. Together with Proposition \ref{bound outer-inner radius}, this yields
\begin{eqnarray}\label{H upper1}
    H(\cdot, t) \leq \frac{R_0^{\frac 1 \a} }{\ve}\max\left\{ 
\max_M H(\cdot,0), \frac{2(1+\a)}{\a \ve^{\frac 1 \a}} \right\},
\end{eqnarray}for $t\in \left[0,\min\left\{ \tau_0 , T^*\right\}\right)$. 

Now we take any $s>0$ and deduce from \eqref{ODE Phi} that, whenever $\wt\Phi>\left(\frac{2n(1+\a)}{\a \ve^{1+\frac 1 \a}}\right)^\a$, we have \begin{eqnarray*}\label{ODE Phi2}
    \frac{d}{dt} \wt{\Phi}\leq  -(1+\a) \wt \Phi^2 , ~~\text{ a.e. } t \in [s,\max\{s+\tau_0,T^*\}).
\end{eqnarray*}This implies
\begin{eqnarray*}
    \wt\Phi(t) \leq \frac 1 {\wt\Phi^{-1}(s)+(1+\a) (t-s)}\leq \frac 1 {(1+\a)(t-s)},
\end{eqnarray*}
thus it follows
\begin{eqnarray*}
    \Phi(\cdot, t)\leq \max\left\{  \left(\frac{2n(1+\a)}{\a \ve^{1+\frac 1 \a}}\right)^\a, \frac 1 {(1+\a)(t-s)} \right\},
\end{eqnarray*}for $t\in [s,\min\{s+\tau_0,T^*\})$. In particular, if $t>s+\frac{\tau_0}{2}$, we obtain
\begin{eqnarray*}
    \Phi(\cdot, t)\leq \max\left\{  \left(\frac{2n(1+\a)}{\a \ve^{1+\frac 1 \a}}\right)^\a, \frac 2 {(1+\a)\tau_0 } \right\},
\end{eqnarray*} and so
\begin{eqnarray}\label{H upper2}
    H(\cdot, t) \leq R_0^{\frac  1 \a} \max\left\{  \frac{2n(1+\a)}{\a \ve^{1+\frac 1 \a}} , \left(\frac 2 {(1+\a)\tau_0} \right)^{\frac 1 \a} \right\},
\end{eqnarray}for $t\in [s+\frac {\tau_0} 2,\min\{s+\tau_0,T^*\})$. Since $s$ is arbitrary, the conclusion \eqref{H upper bdd} follows by combining \eqref{H upper1} and \eqref{H upper2}.
\end{proof}

Proposition \ref{H upper bdd prop} and Proposition \ref{convexity pres} immediately imply that all principal curvatures of $\S_t$ are bounded.

\begin{cor}\label{curvature bound}  Let $\S_t$ be the solution of the flow  \eqref{MCF-capillary}. Then there exists $C>0$ depending only on $\S_0$ such that the the principal curvatures of $\S_t$ satisfy
$$\max\limits_{1\leq i \leq n} \kappa_i  \le C, ~~~\forall ~ t\in [0, T^*).$$
\end{cor}

From the upper bound of the mean curvature, we obtain the following estimate for the non-local term $\phi(t)$ of flow \eqref{MCF-capillary}.
\begin{prop}\label{phi bounds}
    Along the flow \eqref{MCF-capillary}, there holds
    \begin{eqnarray}\label{global term bound}
    b\leq   \phi(t)\leq b^{-1},
    \end{eqnarray}for some positive constant $b$, which only depends on the initial datum.
\end{prop}

\begin{proof}
The upper bound in \eqref{global term bound} follows directly from Proposition \ref{bound outer-inner radius} and Proposition \ref{H upper bdd prop}.

For the proof of the lower bound, we will need the following Minkowski-type inequality for capillary hypersurfaces in the half-space (see Theorem 1.2 in \cite{WWX23} and the following remarks)
\begin{eqnarray}\label{mink ineq1} 
    \int_\S H(1-\cos\theta \<\nu,E_{n+1}\> ) dA \geq c_{n,\theta}W_{\theta}(\wh\S)^{\frac{n-1}{n}},\end{eqnarray}where $c_{n, \theta} :=  n(n+1)^{\frac 1 n}|\wh{\C_\theta}|^{\frac 1 n}.$ 
\

Let us first look at the case that $\phi(t)$ is chosen as in \eqref{phi-volume pres}. 
If $\a\geq 1$, by the H\"{o}lder inequality and Minkowski inequality \eqref{mink ineq1}, we have

\begin{eqnarray*}
\lefteqn{c_{n,\theta} W_{\theta}(\wh{\S_t})^{\frac {n-1}{n}}  \leq   \int_{\S_t} H(1-\cos\theta \<\nu,E_{n+1}\>) dA}
\\& \leq & \left( \int_{\S_t} H^\a (1-\cos\theta \<\nu,E_{n+1}\> ) dA\right)^{\frac 1 \a} \left(\int_{\S_t}  (1-\cos\theta \<\nu,E_{n+1}\>) dA \right)^{1-\frac 1 \a}
 \\&=& \phi^{\frac 1 \a}  \int_{\S_t}(1-\cos\theta \<\nu,E_{n+1}\>)dA,
\end{eqnarray*} which implies
\begin{eqnarray*}
    \phi(t) \geq c_{n,\theta}^\a W_{\theta}(\wh{\S_t})^{-\frac \a n}.
 \end{eqnarray*}Taking into account Proposition \ref{bound outer-inner radius}, the conclusion follows.

If $0<\a<1$, we obtain from Minkowski inequality \eqref{mink ineq1} and Proposition \ref{H upper bdd prop},
\begin{eqnarray*}
c_{n,\theta} W_{\theta}(\wh{\S_t})^{\frac {n-1}{n}} &\leq &   \int_{\S_t} H(1-\cos\theta \<\nu,E_{n+1}\>)dA  \\& \leq & \max_{\S_t} H^{1-\a} \int_{\S_t} H^\a (1-\cos\theta \<\nu,E_{n+1}\>) dA\\&  \leq & C_1\int_{\S_t} H^\a (1-\cos\theta \<\nu,E_{n+1}\>) dA.
\end{eqnarray*}Together with Proposition  \ref{bound outer-inner radius}, this implies
\begin{eqnarray*}
    \phi(t)\geq \frac{ c_{n,\theta}}{C_1} W_{\theta}(\wh{\S_t})^{-\frac 1 n}\geq b,
\end{eqnarray*}for some positive constant $b$, depending on the initial datum.

In the case $\phi(t)$ is chosen as in \eqref{phi-area pres}, for any $\a>0$, by H\"older inequality,
\begin{eqnarray*}
& & \int_{\S_t} H(1-\cos\theta \<\nu,E_{n+1}\>) dA  \\
& \leq & \left( \int_{\S_t} H^{\a+1} (1-\cos\theta \<\nu,E_{n+1}\> ) dA\right)^{\frac 1 {\a+1} } \left(\int_{\S_t}  (1-\cos\theta \<\nu,E_{n+1}\>) dA \right)^{\frac{\a}{\a+1}}
 \\&=& \phi^{\frac 1 {\a+1}}  \left(\int_{\S_t} H(1-\cos\theta \<\nu,E_{n+1}\>) dA\right)^{\frac 1 {\a+1}} W_{\theta}(\wh{\S_t})^{\frac \a {\a+1}}.
\end{eqnarray*}
Using again the Minkowski inequality \eqref{mink ineq1}, 
\begin{eqnarray*}
    \phi(t) &\geq & \left(\int_{\S_t} H(1-\cos\theta \<\nu,E_{n+1}\>) dA\right)^\a W_{\theta}(\wh{\S_t})^{-\a}
    \\& \geq & c_{n,\theta}^\a  W_{\theta}(\wh{\S_t})^{-\frac \a n},
\end{eqnarray*}and we obtain the desired lower bound for $\phi$ by Proposition \ref{bound outer-inner radius}.
\end{proof}

Before proving a uniform lower bound on the mean curvature, we need to estimate the position of our evolving hypersurface by showing that it cannot drift arbitrarily far during the flow. For constrained flows of closed hypersurfaces, McCoy has used an Alexandrov-type reflection argument due to Chow and Gulliver to show that the solution is contained for all times in a suitable fixed ball, see e.g. \cite[Proposition 3.4]{Mc03}.

Here we use a similar approach by using a different reflection argument, again due to Chow and Gulliver. We adapt the notation of \cite[Section 2]{Chow2} to describe the reflection of a capillary hypersurface across a vertical hyperplane. Given $V\in\mathbb{S}^{n}$ such that $\langle V,E_{n+1} \rangle=0$ (equivalently, $V \in \mathbb{S}^{n} \cap \p \RR^{n+1}_+$) and given $c\in \mathbb{R}$, we consider the hyperplane $\Pi_V^c=\{x \in \mathbb{R}^{n+1}\,|\, \langle x,V\rangle=c\}.$
 	Let $H^+(\Pi_V^c)$ (resp. $H^-(\Pi_V^c)$) be the halfspace $\{y\in \mathbb{R}^{n+1}\,|\,\langle x,V \rangle > c\}$ (resp. $\{x\in \mathbb{R}^{n+1}\,|\,\langle x,V \rangle < c\}$).

Let $\Sigma \subset \ol{\RR^{n+1}_+}$ be a capillary hypersurface  which bounds the domain $\widehat \Sigma$ and let $\Sigma^{\Pi}$ be the reflection of $\S$ about $\Pi_V^c$, i.e. $\Sigma^{\Pi}=\{x-2(\langle x,V \rangle-c)V\,|\,x\in \Sigma\}$. Clearly, $\Sigma^{\Pi}$ is again a capillary hypersurface with the same contact angle in $\ol{\RR^{n+1}_+}$
 	 We say that $\Sigma$ can be strictly reflected at $\Pi_V^c$ if $\Sigma^{\Pi}\cap H^-(\Pi_V^c) \subset (\widehat \Sigma \setminus \Sigma) \cap H^-(\Pi_V^c)$ and $V$ is not tangent to $\Sigma$ at the points in $\Sigma \cap \Pi_V^c$. 
	 
	 Then we have the following result. 
	
		\begin{thm}\label{ChowGulliver}
	Let $\Sigma_t$, for $t \in [0,T^*)$, be a smooth family of capillary hypersurfaces solving the flow \eqref{MCF-capillary}. Let $V\in\mathbb{S}^{n} \cap \p \RR^{n+1}_+$ and let $c\in \mathbb{R}$. If $\Sigma_0$ can be reflected strictly at $\Pi_V^c$, then $\Sigma_t$ can be reflected strictly at $\Pi_V^c$ for all time $t \in [0,T^*)$.
	\end{thm}

We just outline the proof of Theorem \ref{ChowGulliver}, which is almost the same as \cite[Theorem 2.2]{Chow2}, by using standard arguments of the Alexandrov moving plane method. The possibility that $\Sigma_t^{\Pi}$ and $\Sigma_t$ first touch at a point of the capillary boundary on $\p\RR^{n+1}_+$ is ruled out by the same contact angle condition of $\S_t^\Pi$ and $\S_t$ , similar to the case where $V$ becomes tangent to $\Sigma_t$ at a point in $\Sigma_t \cap \Pi_V^c$. We observe that the maximum principle argument used in the proof is not affected by the presence of the nonlocal term, since $\phi(t)$ has the same value for the original hypersurface and for the reflected one.

It is interesting to note that the convexity of $\Sigma_t$ is not required for Theorem \ref{ChowGulliver} and that it suffices to assume $H>0$; in addition, the contact angle can be an arbitrary value $\theta \in (0,\pi)$.

\begin{rem}\label{corCG}
Observe that, if a capillary hypersurface $\Sigma$ is contained in the halfspace  $H^-(\Pi_V^c)$, then $\Sigma^{\Pi}\cap H^-(\Pi_V^c)= \emptyset$ and it trivially satisfies the strict reflection property. Conversely, if $\Sigma \subset H^+(\Pi_V^c)$, it necessarily violates the reflection property. Therefore Theorem \ref{ChowGulliver} has the following corollary: if the initial hypersurface satisfies $\Sigma_0 \subset H^-(\Pi_V^c)$ for some vertical hyperplane $\Pi_V^c$, then we cannot have $\Sigma_t \subset H^+(\Pi_V^c)$, for any $t>0$.
\end{rem}

Now we are ready to show the solution of flow \eqref{MCF-capillary} remains inside a suitably large spherical cap with a fixed center point for all time.

 \begin{lem}\label{enclosed cap}
 Let $\S_t$ be the solution of \eqref{MCF-capillary}. Then there exists $z^* \in \p \RR^{n+1}_+$ and $R^*>0$ 
 such that
     \begin{eqnarray}\label{out cap enclose}
         \wh{\S_t} \subset  \wh{\C_{{R^*}, \theta}(z^*)}
     \end{eqnarray}for all $t\in [0,T^*)$.
 \end{lem}

\begin{proof}
Let us fix any $R> R_0$, with $R_0$ given by Proposition \ref{bound outer-inner radius}, and let  $\C_{R,\theta}(z^*)$ be a spherical cap enclosing $\Sigma_0$ in its interior. For simplicity of notation, we assume that $z^*=0$. Then, for any horizontal unit vector $V$ and $X \in \Sigma_0$ we have $V \cdot X < R \sin \theta$ and so $\Sigma_0 \subset H^-(\Pi_V^{c_0})$ with $c_0=R \sin \theta$.

Given any time $t>0$, Proposition \ref{bound outer-inner radius} shows that $\Sigma_t \subset \wh{\C_{R,\theta}(z)}$ for a suitable spherical cap centered at some point $z=z(t) \in \p \RR^{n+1}_+$. Then, for any horizontal unit vector $V$ and $X \in \Sigma_t$ we have $$V \cdot X > V \cdot z -R \sin \theta,$$ and therefore $\Sigma_t \subset H^+(\Pi_V^{c})$ for $c=V \cdot z -R \sin \theta$. By Remark \ref{corCG}, this implies that $c<c_0$, that is, $$V \cdot z < 2 R \sin \theta.$$ Since $V \in \p \RR^{n+1}_+$ is an arbitrary unit vector, this implies that \begin{eqnarray}\label{bound of z}
|z| \leq 2 R \sin \theta.
\end{eqnarray}
Let us now set $R^*:=\left(1+\frac{2 \sin \theta}{1 - \cos \theta} \right)R$. In view of $ \wh{\Sigma_t} \subset \wh{\C_{R,\theta}(z)}$ and \eqref{bound of z}, we find, for any $X \in \Sigma_t$,
\begin{eqnarray*}
|X+\cos \theta R^* E_{n+1}| & \leq & |X-(z-\cos \theta R E_{n+1})| + |z| + (R^*-R)\cos \theta \\
& \leq & R+ 2 R \sin \theta + (R^*-R)\cos \theta \\
& = & R \left(1+2 \sin \theta +\frac{2 \sin \theta \cos \theta}{1 - \cos \theta} \right) =R^*,
\end{eqnarray*}
showing that $\wh{\Sigma_t} \subset \wh{\C_{{R^*},\theta}(0)}$ for all $t>0$.  
\end{proof}

To obtain a lower bound for $H$,  we adapt the method in \cite[Section 4.1]{BS}, which reverses the sign of the test function in  \eqref{test-Phi}. This idea had been previously used in \cite[Lemma 4.3]{S06} by Schn\"urer in the different setting of a local expanding curvature flow.
 The lower bound will ensure that the flow \eqref{MCF-capillary} is strictly parabolic uniformly in time. In contrast with the proof the upper bound in Proposition \ref{H upper bdd prop}, the nonlocal term $\phi$ in the equation plays here a crucial role and Proposition \ref{phi bounds} is essential.
\begin{prop}\label{H lower bound prop}
Let $\S_t$ be a solution of the flow \eqref{MCF-capillary}. Then there holds 
    \begin{eqnarray*}\label{H lower bdd}
        H  \geq C \mbox{ on }\S_t,~~~\forall ~ t\in [0, T^*),
    \end{eqnarray*} where $C$ is a positive constant, depending on the initial datum.
\end{prop}

\begin{proof} 

For any time $t\geq 0$, let $z^*$ be chosen as in Lemma \ref{enclosed cap}. Then  we know that
\begin{eqnarray*}
 0 \leq \bar u:=\frac{\<X-z^*,\nu\>}{1-\cos\theta \<\nu, E_{n+1}\>}  \leq \frac c 2,
\end{eqnarray*} for some positive constant $c:=2R^*>0$, where $R^*$ is given in Lemma \ref{enclosed cap}. In turn
\begin{eqnarray}\label{bound of bar u}
    \frac{c}{2}\leq c- \bar u\leq c,
\end{eqnarray}for all $t\in [0,T^*)$. We then consider the  function
\begin{eqnarray}\label{test-Psi}
    \Psi :=\frac{H^\a}{c- \bar u},
\end{eqnarray}which is well-defined for  $t\in [0,T^*)$. From \eqref{phi neumann}, 
\begin{eqnarray}\label{neum of Psi}
    \n_\mu \Psi=0 \text{ on } \p M.
\end{eqnarray}
With computations similar to the proof of Proposition \ref{H upper bdd prop}, we have
\begin{eqnarray*}
    \L \Psi&= &(\a H^{\a-1} |h|^2-\Psi) \frac{\phi}{\bar u-c}+ \frac{ c\a H^{2\a-1}|h|^2}{(c-\bar u)^2} -(1+\a)\Psi^2  \\&& +2\a H^{\a-1} \<\n \Psi, \n (1-\cos\theta \<\nu, E_{n+1})\> 
    \\ & &+\frac{2\a(1-\cos\theta \<\nu, E_{n+1}\> H^{\a-1}}{\bar u-c} \<\n \Psi, \n \bar u\>
    \\& \geq & (\a H^{\a-1} |h|^2-\Psi) \frac{\phi}{\bar u-c}-(1+\a)\Psi^2 , \text{ mod } \n \Psi. 
\end{eqnarray*}
If the minimum value of $\Psi$ is reached at $t=0$, then we are done. Otherwise, from $\n_{\mu}\Psi= 0$ on $\p M$ in \eqref{neum of Psi} and the Hopf boundary point lemma, we have that $\Psi$ attains its minimum value at some interior point, say $p_{0}\in {\rm int}(M)$. Let us choose $\varepsilon_0>0$ sufficiently small so that  $0<H\leq \varepsilon_0$ implies   
\begin{eqnarray}\label{small H}
    \a H\leq \frac{1}{2c},\quad  2(1+\a) H^\a\leq \frac b 4,
\end{eqnarray}where $b>0$ is the uniform lower bound of $\phi(t)$ in Proposition \ref{phi bounds}. Let us now suppose that $H(p_0) \leq \varepsilon_0$. From the evolution equation for $\L \Psi$,  and using $|h|^2\leq H^2$, Proposition \ref{phi bounds},  \eqref{enclosed cap}, \eqref{bound of bar u} and \eqref{small H}, we have at $p_0$,  
\begin{eqnarray*}
    0 &\geq& \L \Psi \geq \frac{\phi \Psi}{c-\bar u}- \frac{\a H^{\a+1}\phi}{c-\bar u}-(1+\a)\Psi^2 \quad \text{ mod } \n \Psi 
    \\&\geq &  \Psi \left[ \phi \left( \frac 1 c-\a H \right)-(1+\a) \Psi  \right] 
\\& \geq & \Psi \left[ \frac{b}{2c}- 2(1+\a)  \frac{H^\a}{c}\right]>0,
\end{eqnarray*}which is a contradiction. Hence we conclude that
\begin{eqnarray*}
    \Psi \geq \min \left\{ \Psi(\cdot, 0), \frac{\varepsilon_0^\a}{c} \right\}.
\end{eqnarray*}Together with \eqref{bound of bar u}, this yields the uniform lower bound for $H$. 

\end{proof}
 
\subsection{Proof of Theorem \ref{thm1}} \ 

\begin{thm}\label{unif est}    The solution of flow \eqref{MCF-capillary} exists for all time and satisfies uniform $C^\infty$-estimates.\end{thm}

\begin{proof}
    In order to show the long-time existence of flow \eqref{MCF-capillary}, it is convenient to represent the convex hypersurface $\S_t$ as the radial graph over semi-sphere $\bar\SS^n_+$ for some function $\rho(x,t)$ with $(x,t)\in \bar\SS^n_+\times[0, T^*)$, and obtain uniform estimates on the derivatives of $\rho$, see for instance \cite[Section 4]{WWX22} or \cite[Section 3.4]{WeX21}.  Then equation \eqref{MCF-capillary} reduces to a scalar parabolic equation for the function $\rho$ with a capillary boundary value condition on $\p\SS^n_+$. From Proposition \ref{bound outer-inner radius} and Corollary \ref{curvature bound}, we know that $\rho$ is uniformly bounded in $C^2(\bar\SS^n_+\times [0,T^*)$. The corresponding scalar flow for $\rho$  is uniformly parabolic (due to Proposition \ref{H lower bound prop} and Proposition \ref{H upper bdd prop}) and the boundary value condition of $\rho$ satisfies a uniformly oblique property due to $|\cos\theta|<1$. From standard parabolic theory (see e.g. \cite[Theorem 6.1, Theorem 6.4 and Theorem 6.5]{Dong}, also  \cite[Theorem 14.23]{Lie}), we obtain uniform $C^\infty$-estimates for $\rho$. In turn, the long-time existence of solution \eqref{MCF-capillary} follows.
\end{proof}

Finally, we conclude the proof of Theorem \ref{thm1}.
\begin{proof}[\textbf{Proof of Theorem \ref{thm1}}]\ 

For the convergence of the flow \eqref{MCF-capillary}, we can argue using the same way as \cite[Section 4]{S15} (or \cite[Section 4.2]{BS}). 
We exploit the monotonicity of $W_{\theta}(\widehat{\S_t})$ (resp. of $|\wh{\S_t}|$) along our flow. Taking into account the uniform estimates of Theorem \ref{unif est}, we see that the time derivative of $W_{\theta}(\widehat{\S_t})$ (resp. of $|\widehat{\S_t}|$) must tend to zero as $t \to +\infty$. From the explicit expression of these derivatives in the proof of Proposition \ref{cap isoper ratio}, we see that this implies
  \begin{eqnarray*}
        \lim_{t\to+\infty} \max_{\S_t} |H(\cdot,t)-q(t)|=0,
    \end{eqnarray*}
with $q(t)$ as in the proof of Proposition \ref{cap isoper ratio}, which in turn yields
  \begin{eqnarray*}
        \lim_{t\to+\infty} \max_{\S_t} |H^\a(\cdot,t)-\phi(t)|=0.
    \end{eqnarray*}
Therefore, using from Theorem \ref{unif est} and compactness, we see that any possible limit of subsequences of $\S_t$ has constant mean curvature with capillary boundary. 
Combining with the conclusion in \cite[Corollary 1.2]{JWXZ} (see also \cite{Wente}) we know that the limit is a spherical cap, with radius uniquely determined by the constraint on the volume (resp. on the capillary area). By a standard procedure, cf. \cite[Section 4.5]{WWX22} or \cite[Section 3.4]{WeX21}, one can show that since any limit of a convergent subsequence is uniquely determined, then the whole family $\S_t$ smoothly converges to a spherical cap.
\end{proof}

\

\noindent\textit{Acknowledgments:} This work was partially supported by: MIUR Excellence Department Project awarded to the Department of Mathematics, University of Rome Tor Vergata, CUP E83C18000100006  and MUR Excellence Department Project MatMod@TOV  CUP E83C23000330006. C.S. was partially supported by the MUR Prin 2022 Project  "Contemporary perspectives on geometry and gravity"
CUP E53D23005750006 and by the Project "ConDiTransPDE" of the University of Rome "Tor Vergata" CUP E83C22001720005. C.S.
is a member of the group GNAMPA of INdAM.
 L.W. would also like to express his sincere gratitude to Prof. ssa G. Tarantello for her constant encouragement and support.

\end{document}